\documentclass{amsart}
\usepackage{amsmath}
\usepackage{lmodern}
  \usepackage{paralist}
  \usepackage{graphics} 
  \usepackage{epsfig} 
\usepackage{graphicx}  \usepackage{epstopdf}
 \usepackage[colorlinks=true]{hyperref}
\hypersetup{urlcolor=blue, citecolor=red}

\newtheorem{thm}{Theorem}[section]
\newtheorem{prop}[thm]{Proposition}
\newtheorem{cor}[thm]{Corollary}
\newtheorem{lem}[thm]{Lemma}
\newtheorem{hyp}[thm]{Hypothesis}
\newtheorem{defi}[thm]{Definition}
\newtheorem{remarque}[thm]{Remark}
\newtheorem{nota}[thm]{Notation}

\title[Trend to the equilibrium for the Fokker-Planck-Magnetic system] 
      {Trend to the equilibrium for the Fokker-Planck system with a strong external magnetic field}

\author[Zeinab karaki]{}

\subjclass{Primary: 47D06, 35Q84; Secondary: 35P15, 82C40.}
 \keywords{return to equilibrium; hypocoercivity; Fokker-Planck equation ; magnetic field;
enlargement space.}

 \email{zeinab.karaki@univ-nantes.fr}

\begin{document}
\maketitle

\centerline{\scshape Zeinab Karaki}
\medskip
{\footnotesize
 \centerline{Universit\'e de Nantes}
   \centerline{Laboratoire de Mathematiques Jean Leray}
   \centerline{ 2, rue de la Houssini\`ere}
    \centerline{BP 92208 F-44322 Nantes Cedex 3, France}
} 

\medskip

\bigskip


\begin{abstract}
We consider the Fokker-Planck equation with a strong external magnetic field. Global-in-time solutions are built near the Maxwellian, the global equilibrium state for the system. Moreover, we prove the convergence to equilibrium at exponential rate. The results are first obtained on spaces with an exponential weight. Then they are extended to larger functional spaces, like the Lebesgue space and the Sobolev space with polynomial weight, by the method of factorization and enlargement of the functional space developed in [Gualdani, Mischler, Mouhot, 2017].
\end{abstract}

\tableofcontents

\section{Introduction and main results}
\subsection{Introduction}
In this article, we are interested in inhomogeneous kinetic equations. These equations model the dynamics of a charged particle system described by a  probability density  $ F (t, x, v) $ representing at time $ t  \geq 0 $ the density of particles at position $ x \in \mathbb {T}^{3} $ and at velocity $ v \in \mathbb {R}^{3} $. \\
~\par In the absence of force and collision, the particles move in a straight line at constant speed according to the principle of Newton, and $ F $ is the solution of the Vlasov equation
$$ \partial_{t} F + v \cdot \nabla_{x} F = 0, $$
where $ \nabla_{x} $ is the gradient operator with respect to the variable $ x $, and the symbol  <<$\cdot$>> represents the scalar product in the Euclidean space $ \mathbb{R}^{3} $. When there are forces and shocks, this equation must be corrected. This leads to various kinetic equations, the most famous being those of Boltzmann, Landau and Fokker-Planck. The general model for the dynamics of the charged particles, assuming that they undergo shocks modulated by a collision kernel $ Q $ and under the action of an external force $ \mathcal{F} \in \mathbb{R}^{3} $, is written by the following kinetic equation:
\begin{align}
\partial_{t}F + v \cdot \nabla_{x}F + \mathcal{F}(t,x)\cdot  \nabla_{v}F = Q(F),
\end{align}
where $ Q $, possibly non-linear, acts only in velocity and where $ \mathcal{F} $ can even depend on $ F $ via Poisson or Maxwell equations. \\
~\par  According to the $H$-theorem of Boltzmann in 1872, there exists a quantity  $ H (t) $ called entropy which varies monotonous over time, while the gas relaxes towards the thermodynamic equilibrium characterized by the Maxwellian: it is a solution time independent of  equation $ (1) $  having the same mass as the initial system. The effect of the collisions will bring the distribution $ F (t) $ to the Maxwellian with time. A crucial question is then to know the rate of convergence and this question has been widely studied over the past 15 years, in particular with the so called hypocoercive strategy (see \cite{villani2006hypocoercive} or \cite{herau2017introduction} for an introductive papers).

\subsection{Fokker-Planck equation with a given external magnetic field}
\subsubsection{Presentation of the equation}
We are interested in the Fokker-Planck inhomogeneous linear kinetic equation with a fixed external magnetic field $ x\mapsto B_{e}(x) \in \mathbb{R}^{3} $ which depends only on the spatial variables $ x \in  \mathbb{T}^{3} \equiv [0,2 \pi]^{3} $. The Cauchy problem  is the following:
\begin{equation}
\begin{cases}
\partial_{t} F + v \cdot \nabla_{x} F - (v \wedge B_{e})\cdot \nabla_{v}F= \nabla_{v}\cdot (\nabla_{v} + v) F \\
F(0,x,v)=F_{0}(x,v),
\end{cases}
\label{1}
\end{equation} 
Here $ F $ is the distribution function of the particles, and represents the  density of probability of  presence of particles at time $ t \geq 0 $ at the position $ x \in \mathbb{T}^{3} $ and with a speed $ v \in \mathbb{R}^{3} $. (Where <<$\wedge$>> indicates the vector (cross) product.)
\\
~\par
We define the Maxwellian 
$$ \mu (v): = \frac{1}{(2 \pi)^{3/2}} \, e^{- v^{2} / 2}. $$
It is the (only) time independent solution of the system \eqref{1}, since
$$\partial_{t} \mu + v \cdot \nabla_{x} \mu=0, \quad\nabla_{v} \cdot(-\nabla_{v} + v) \mu =0 \quad\text{ and } \quad(v \wedge B_{e})\cdot \nabla_{v}\mu =0 .$$
Concerning \eqref{1}, we are interested in the return to the global equilibrium $ \mu $ and the convergence  of $ F $ to $ \mu $ in norms $ L^{2} (dx d \mu) $ and $ H^{1} (dxd \mu) $  defined by
\begin{align*}
\forall h\in L^{2}(dxd\mu), \quad \Vert h\Vert_{ L^{2}(dxd\mu)}^{2}&=\iint_{\mathbb{T}^{3}\times\mathbb{R}^{3}}\, \vert h(x,v)\vert^{2}\,dxd\mu,\\
\forall h\in H^{1}(dxd\mu), \quad \Vert h\Vert_{H^{1}(dxd\mu)}^{2}&=\Vert h\Vert_{L^2 (dxd\mu )}^{2}+\Vert \nabla_{x}h\Vert_{L^2 (dxd\mu )}^{2}+\Vert \nabla_{v}h\Vert_{L^2 (dxd\mu )}^{2},
\end{align*}
where $ d \mu: = \mu (v) dv $ and the (real) Hilbertian scalar product $ \langle.,. \rangle $ on the space $ L^{2} (dxd \mu) $ defined by
$$ \forall g,h \in L^{2} (dxd \mu), \quad \langle h , g \rangle = \iint hg \, dxd \mu. $$
To answer such questions, when $ F $ is close to the equilibrium $ \mu $, we define  $ f $ to be the standard perturbation of $ F $ defined by
$$ F=\mu + \mu f .$$
We then rewrite equation \eqref{1} in the following form:
\begin{equation}
\begin{cases}
\partial_{t} f + v \cdot \nabla_{x} f - (v \wedge B_{e})\cdot \nabla_{v}f=-(- \nabla_{v} + v  )\cdot \nabla_{v}f  \\
f(0,x,v)=f_{0}(x,v)
\end{cases}
\label{2}
\end{equation}
We now introduce the main assumption on $B_e$.
\begin{hyp} \label{hyp 1}
The magnetic field $ B_{e} $ is indefinitely derivable on $ \mathbb{T}^{3} $.
\end{hyp} 
\subsubsection{The main results}
First we will show that the problem \eqref{2} is well-posed in the  $ L^{2} (dxd \mu) $ space, in the sense of the associated semi-group (See \cite{pazy2012semigroups}).
  We associate with the problem \eqref {2} the  operator $ P_{1} $ defined by
  \begin{align}
&P_{1}:= X_{0} - L ,\label{18}\\
\text{ where } &X_0= v \cdot \nabla_{x}  - (v \wedge B_{e})\cdot \nabla_{v} \text{ and } L= (-\nabla_{v} + v)\cdot\nabla_{v}  \label{18bis}
\end{align}
The problem \eqref{2} is then written
\begin{equation}
\begin{cases}
\partial_{t}f + P_{1}f =0 \\
f(0,x,v)=f_{0}(x,v)
\end{cases}
\label{5}
\end{equation}
\begin{thm} \label{thm 1}
Under Hypothesis \ref{hyp 1} and with $ f_{0} \in C^{\infty}_{0}(\mathbb{T}^{3}\times \mathbb{R}^{3}) $, the problem \eqref{5} admits a unique solution
  $ f \in C ^{1} ([0, + \infty [, L^{2} (dxd \mu)) \cap C ([0, + \infty [, C^{\infty}_{0}(\mathbb{T}^{3}\times \mathbb{R}^{3})) $ .
\end{thm}
We also show the exponential convergence towards equilibrium in the norm $ L^{2} (dx d \mu) $.

\begin{thm} \label{thm 2}
Let $ f_{0} \in L^{2} (dxd \mu) $ such that $ \langle f_{0} \rangle = \displaystyle \iint f_0 (t, x, v) \, dxd \mu = 0 $. If $ B_{e} $ satisfies  Hypothesis \ref{hyp 1}, then there exist $ \kappa > 0 $ and $ c> 0 $ (two explicit constants  independent of $ f_{0} $) such that
$$\forall t \geq 0, \quad  \Vert f(t) \Vert_{L^2 (dxd\mu )} \leq c e^{-\kappa t} \Vert f_{0} \Vert_{L^2 (dxd\mu ) }.$$
\label{4}
\end{thm}
Note that in the preceding statement the mean $ \langle f_ {0} \rangle$ is preserved over time.
~\par 
We give a result about the return to the global equilibrium $ \mu $ with an exponential decay rate in the space $ H^1 (dxd \mu) $.
\begin{thm} \label{thm 4}
  There exist $ c, \kappa > 0 $ such that $ \forall f_{0} \in H^{1} (dxd \mu) $ with $ \langle f_{0} \rangle = 0 $, the solution $ f  $ of the system \eqref{2} satisfies
  \begin{align*}
 \forall t \geq 0, \quad \Vert f(t) \Vert_{H^{1}(dxd\mu)} \leq c e^{-\kappa t} \Vert f_{0} \Vert_{H^{1}(dxd\mu)}.
 \end{align*}
\end{thm}
We are interested in extending the results about the exponential decay of the semi-group to much larger spaces, following the work of Gualdani-Mischler-Mouhot in \cite{gualdani2010factorization}.
  The following result gives convergence in $ L^{p}(m) $ norms of $ F $ to $ \mu $, where the space $ L^{p} (m) $ for $ p \in [1,2] $ is the  weighted Lebesgue space associated with the norm
  \begin{align*}
\Vert F \Vert_{L^{p}(m)} := \Vert Fm\Vert_{L^{p}} = \left( \int_{\mathbb{R}^{3}\times \mathbb{T}^{3}}\, F^{p}(x,v)\,m^{p}(v) dv dx\right)^{\frac{1}{p}},
\end{align*}
for some given weight function
$ m = m (v)> 0 $ on $ \mathbb{R}^3 $. Since there is no ambiguity we again denote 
$$\langle H\rangle=\iint H \,dxdv,$$ the mean with respect to the usual $L^1$ norm.
The main result of this paper in this direction is the following.
\begin{thm}\label{thm 5}
Let $p\in [1,2]$, let $m=\langle v \rangle^{k}:=(1+\vert v\vert^2 )^{k/2}$,  $k>3(1-\displaystyle \frac{1}{p})$ , and assume Hypothesis \ref{hyp 1} . Then for all $ 0>a >3(1-\displaystyle\frac{1}{p})-k$ and for all $F_{0}\in L^{p}(m)$, there exists $c_{k,p}>0$ such that the solution $ F$ of the problem \eqref{1} satisfies the decay estimate
\begin{align}
 \forall t \geq 0, \quad \Vert F(t) - \mu\, \langle F_{0}\rangle \Vert_{L^{p}(m)} \leq c_{k,p}e^{at}\,  \Vert F_{0} - \mu \, \langle F_{0}\rangle \Vert_{L^{p}(m)}.
 \label{14}  
 \end{align}
 \end{thm}
 It is also possible to obtain the same type of results in weighted  Sobolev space $ \tilde{W}^{1, p} (m) $ which is defined by 
 $$ \tilde{W}^{1,p}(m) = \{ h\in L^p (m) \text{ such that } \langle v\rangle h , \nabla_{v}h \text{ and } \nabla_{x} h \in L^{p}(m) \}. $$
We equip the previous space with the following standard norm:
\begin{align}\label{6bis}
\Vert h \Vert_{\tilde{W}^{1,p}(m)} = \left( \Vert h\Vert_{L^{p}(m\langle v\rangle)}^{p}+ \Vert \nabla_{v}h\Vert_{L^{p}(m)}^{p}+ \Vert \nabla_{x}h\Vert_{L^{p}(m)}^{p}\right)^{\frac{1}{p}}.
\end{align}

\begin{hyp} \label{hyp 2}
Let $ p \in [1,2] $, the polynomial weight $m(v)=\langle v\rangle^{k}$ is such that 
\begin{align}
 k>3(1-\frac{1}{p})+\frac{7}{2}+ \max\left(\Vert B_e\Vert_{L^{\infty}(\mathbb{T}^{3})},\frac{1}{2}\Vert \nabla_{x}B_e\Vert_{L^{\infty}(\mathbb{T}^{3})}\right). 
\label{15} 
\end{align}
\end{hyp}
The second main result of this paper is the following.
\begin{thm}\label{thm 6}
Let $ m $ be a weight that satisfies Hypothesis \ref{hyp 2} with $ p \in [1,2] $ and assume Hypothesis \ref{hyp 1}. If $ F_{0} \in \tilde{W}^{1, p} (m) $, then there is a solution $ F  $ of the problem \eqref{1}, such that $ F(t)  \in \tilde{W}^{1, p} (m) $ for all $t\geq 0$, and it satisfies the following decay estimate:
\begin{align}
\forall t\geq 0,\quad \Vert F(t) -\mu \langle F_{0}\rangle \Vert_{\tilde{W}^{1,p}(m)} \leq C e^{at}\, \Vert F_{0}-\mu\langle F_{0}\rangle \Vert_{\tilde{W}^{1,p}(m)}
 \label{eg 8}
\end{align}
with $0>a>\max (a_{m,1}^{i}, a_{m,2}^{i}, -\kappa), i\in \{1,2,3\}$, where   $a_{m,1}^{i}$ and $a_{m,2}^{i}$  are functions defined afterwards in \eqref{def 27}-\eqref{def 29} and  \eqref{def 21}-\eqref{def 23} and $\kappa$ is defined in Theorem \ref{thm 4}.
\end{thm}
We will end this part by a brief review of the literature related to the analysis of kinetic PDEs  using hypocoercivity methods. In some studies \cite{herau2004isotropic, nier2005hypoelliptic,villani2006hypocoercive, dric2009hypocoercivity}, the treated hypocoercivity method is very close to that of hypoellipticity following the method of Kohn, which deals simultaneously with regularity properties and trend to the equilibrium.

The hypocoercive results were developped for simple models in \cite{dolbeault2013hypocoercivity, herau2006hypocoercivity, mouhot2006quantitative, dric2009hypocoercivity}, the methods used were close in spirit to the ones developed in Guo \cite{guo2002vlasov,guo2003vlasov} in functional spaces with exponential weights.

In recent years, the theory of factorization and enlargement of Banach spaces was introduced in \cite{gualdani2010factorization} and \cite{mischler2016exponential}. This theory allows us to extend hypocoercivity results into much larger spaces with polynomial weights. We refer for example to \cite{bouin2017hypocoercivity} and \cite{mischler2016exponential}, where the authors show, using a factorization argument, the return to equilibrium with an exponential decay rate for the Fokker-Planck equation with an external electrical potential, or \cite{herau2017cauchy} for the inhomogeneous Boltzmann equation  without angular cutoff case.\\

We conclude this sectio with some comments on our result. For the proof of Theorem \ref{thm 2} , we follow the micro-macro method proposed in \cite{herau2017introduction}. Note that for the proof of Theorem \ref{thm 2}, the black box method proposed in \cite{dolbeault2013hypocoercivity} (see also \cite{bouin2017hypocoercivity}) could perhaps be employed, anyway the presence of the Magnetic field induces same difficulties. 
To prove Theorem \ref{thm 5} and \ref{thm 6}, we apply the abstract theorem of enlargement from \cite{gualdani2010factorization, mischler2016exponential} to our Fokker-Planck-Magnetic linear operator. We deduce the semi-group estimates of Theorem \ref{thm 2} on large spaces like $L^p(\langle v\rangle^k)$ and $\tilde{W}^{1,p}(\langle v\rangle^k)$ with $p\in [1,2]$.

We hope that this first work will help in future investigations of non-linear problems like the Vlasov-Poisson-Fokker-Planck or Vlasov-Maxwell-Fokker-Planck equations (see \cite{guo2002vlasov, herau2016global} and \cite{ guo2003vlasov, yang2010global}).

~\par\textbf{Plan of the paper: } This article is organized as follows. In Section 2, we prove that the Fokker-Planck-magnetic operator $ P_1 $ is a generator of a strongly continuous semi-group over the space $ L^2 (dxd\mu) $. In section 3, we show hypocoercivity in the weighted spaces $ L^2 $ and $ H^1 $ with an exponential weight. Finally, section 4 is devoted to the proofs of Theorems \ref{thm 5} and \ref{thm 6} with factorization and  enlargement of the functional space arguments.

\section{Study of the operator $P_{1}$}
In this part, we show that the problem \eqref{5} is well-posed in the space $ L^2 (dxd  \mu) $ in the sense of semi-groups. By  the Hille-Yosida Theorem, it is sufficient to show that $ P_{1} $ is maximal accretive in the space $ L^2 (dxd \mu) $.
\begin{nota}\label{remark 1}
We define $ P_{0} $ by
$$P_{0}=  v \cdot \nabla_{x}  - (v \wedge B_{e})\cdot \nabla_{v}- \nabla_{v}\cdot (\nabla_{v} + v).$$
The perturbation of the Cauchy  problem \eqref{1} reduces the study of the operator $ P_{1} $  defined in \eqref{18} to the study of $P_0$, since $ P_{1} $ is obtained via a conjugation of the operator $ P_{0} $ by the function $ \mu $, that is to say
$$P_{1}u=(\mu^{-1}\,P_{0}\,\mu) u\quad \forall u\in D(P_{1}). $$
Similarly, we can define the operator $ P_{\theta} $ as the conjugation of the operator $ P_{1} $ by the function $ \mu^{\theta} $ with $ \theta \in ] 0,1] $. Note that any result on the operator $ P_{\theta} $  is also true on the operator $ P_{1} $ in the corresponding conjugated space.
\end{nota}
We will work in this section on operator $ P_{1/2} $ which is defined by
$$ P_{1/2}:=v\cdot \nabla_{x} - (v \wedge B_{e})\cdot \nabla_{v} + (-\nabla_{v} + \frac{v}{2})\cdot(\nabla_{v} + \frac{v}{2}) \\  = X_{0} + b^{*}b,$$
here $b=(\nabla_{v}+\frac{v}{2})$ and $X_0$ is defined in \eqref{18bis}.
We now show that  operator $ P_{1/2} $ is maximal accretive in the space $ L^2 (dxdv) $ and note that this gives the same result for $ P_{1} $ in the space $ L^2 (dxd \mu) $. We study the following problem:
\begin{equation}
\begin{cases}
\partial_{t}u + P_{1/2}u=0 \\
u(0,x,v)=u_{0}(x,v).
\end{cases}
\end{equation}

\begin{prop}\label{prop 2}
Suppose that $ B_{e} $ satisfies the Hypothesis \ref{hyp 1}. Then the closure with respect to the norm $L^2(\mathbb{T}^{3}\times\mathbb{R}^{3})$ of the magnetic-Fokker-Planck operator $\overline{P}_{1/2}$ on the space $C^{\infty}_{0}(\mathbb{R}^{3} \times \mathbb{T}^{3})$ is maximally accretive.
\begin{proof}
We adapt here the proof given in \cite[page 44]{nier2005hypoelliptic}. We apply the abstract criterion by taking $H=L^{2}( \mathbb{T }^{3} \times \mathbb{R}^{3})$
and the domain of $ P_{1/2} $ defined by $D(P_{1/2})=C^{\infty}_{0}( \mathbb{T}^{3}\times \mathbb{R}^{3}). $ First, we show the accretivity of the  operator $ P_{1/2} $. 
When $ u \in D(P_{1/2}) $, we have to show that $\langle P_{1/2}u,u \rangle\geq 0 $. Indeed,
\begin{align*}
\langle\ P_{1/2}u,u \rangle & = \langle v\cdot \nabla_{x} u- (v \wedge B_{e})\cdot \nabla_{v}u + (-\nabla_{v} + \frac{v}{2})(\nabla_{v} + \frac{v}{2})u,u \rangle \\
&= \underbrace{\iint v\cdot \nabla_{x}u\times u \,dx dv }_{=0}-\underbrace{ \iint (v \wedge B_{e})\cdot \nabla_{v}u\times u \,dx dv }_{=0}+ \langle b^{*}b u,u \rangle \\
&=\Vert bu \Vert^{2}\\
&\geq 0,
\end{align*}
since  operators $ (v \wedge B_{e}) \cdot \nabla_{v} $ and $ v\cdot \nabla_{x} $ are skew-adjoint see Lemma \ref{lem 14}.
~\par Let us now show that there exists $ \lambda_{0}> 0 $ such that the operator $$ T = P_{1/2} + \lambda_{0} Id $$ has dense image in $ H $. We take $ \lambda_{0} = \frac{3}{2} +1 $ (following \cite{nier2005hypoelliptic}). 
 Let $u \in H$ satisfy
\begin{align}
\langle u,(P_{1/2} + \lambda_{0} Id)h \rangle =0, \quad \forall h \in D(P_{1/2}).
\label{6}
\end{align}
We have to show that $ u = 0 $. 

First, we observe that equality \eqref{6} implies that
$$ (- \Delta_{v}+ \dfrac{v^{2}}{4} + 1 - X_{0}) u =0,  \text{ in } \mathcal{D'}(\mathbb{R}^{3} \times \mathbb{T}^{3}).$$
Under Hypothesis \ref{hyp 1}, and following Hormander \cite{hormander1967hypoelliptic, hormander1985analysis} or Helffer-Nier \cite[Chapter 8]{nier2005hypoelliptic}, operator $ - \Delta_{v} + \dfrac{v^{2}}{4} + 1 - X_{0} $ is hypoelliptic, so  $u \in C^{\infty}(\mathbb{R}^{3}\times \mathbb{T}^{3})$.\\
Now we introduce the family of truncation functions
$\xi _{k}$ indexed by $k\in \mathbb{N}^{*}$ and defined by \[ \xi_{k} (v):=\xi(\frac{v}{k})\quad \forall k\in \mathbb{N}^*,\]
where $ \xi $ is a $C^{\infty}$ function satisfying  $ 0\leq \xi \leq 1 $, $\xi =1 $ on $B(0,1)$, and $\operatorname{Supp} \xi \subset B(0,2).$\\
For all $ u,w \in C^{\infty}(\mathbb{T}^{3}\times \mathbb{R}^{3})$, we have
\begin{align*}
 \iint & \nabla_{v}(\xi_{k} u)\cdot \nabla_{v}(\xi_{k} w) \,dxdv + \iint \xi_{k}^{2}(\frac{v^{2}}{4} + 1 ) wu \,dxdv + \iint u X_{0}(\xi_{k}^{2} w)\, dxdv \\
= \iint & \vert \nabla_{v} \xi_{k}\vert^{2} w u \,dx dv + \iint (u \nabla_{v}w - w \nabla_{v}u)\cdot \xi_{k} \nabla_{v}\xi_{k} \,dxdv + \langle u, T(\xi^{2}_{k} w)  \rangle.
\end{align*}
When $u$ satisfies \eqref{6} in particular, when  $h=\xi_{k}^{2}w$ , we get for all $ w \in C^{\infty}$
\begin{align*}
 \iint & \nabla_{v}(\xi_{k} u)\cdot \nabla_{v}(\xi_{k} w) dxdv + \iint \xi_{k}^{2}(\frac{v^{2}}{4} + 1 ) w u dxdv + \iint u X_{0}(\xi_{k}^{2} w) dxdv \\
= \iint & \vert \nabla_{v} \xi_{k}\vert^{2} wu dx dv + \iint (u \nabla_{v}w - w \nabla_{v}u)\cdot \xi_{k} \nabla_{v}\xi_{k} dxdv. 
\end{align*}
In particular, we take the test function $ w = u $, so
\begin{align*}
&\langle  \nabla_{v}(\xi_{k} u), \nabla_{v}(\xi_{k} u) \rangle + \iint \xi_{k}^{2}(\frac{v^{2}}{4} + 1 ) u^{2} dxdv + \iint u X_{0}(\xi_{k}^{2} u) dxdv \\
&= \iint \vert \nabla_{v} \xi_{k}\vert^{2} u^{2} dx dv.
\end{align*}
By an integration by parts, we obtain 
\begin{align*}
&\langle  \nabla_{v}(\xi_{k} u), \nabla_{v}(\xi_{k} u) \rangle + \iint \xi_{k}^{2}(\frac{v^{2}}{4} + 1 ) u^{2} dxdv + \iint \xi_{k} u^{2} X_{0}(\xi_{k}) dxdv \\
&= \iint \vert \nabla_{v} \xi_{k}\vert^{2} u^{2} dx dv.
\end{align*}
Which gives the existence of a constant $ c> 0 $ such that,  for all $ k \in \mathbb{N}^{2}$,
\begin{align*}
\Vert \xi_{k} u \Vert^{2} + \frac{1}{4} &\Vert \xi_{k} v u \Vert^{2}\\
&\leq \frac{c}{k^{2}} \Vert u \Vert^{2}  + \frac{c}{k} \Vert (v \wedge B_{e} )\xi_{k} u \Vert \Vert u \Vert.
\end{align*}
This leads to
\begin{align*}
\Vert \xi_{k} u \Vert^{2} + \frac{1}{8}\Vert \xi_{k} v u \Vert^{2}
\leq c(  \frac{1}{k^{2}}+ \frac{c_{\eta}}{k} \Vert B_{e} \Vert_{\infty}^{2}) \Vert u \Vert^{2}+ \eta \Vert \xi_{k} v u \Vert^{2}.
\end{align*}
Choosing $ \eta  \leq \frac{1}{8} $, we get
\begin{align}
\Vert \xi_{k} u \Vert^{2} \leq c(  \frac{1}{k^2}+ \frac{c_{\eta}}{k} \Vert B_{e} \Vert_{\infty}^{2}) \Vert u \Vert^{2}.
\label{16}
\end{align}
Taking $ k \longrightarrow + \infty $ in \eqref{16},  leads to $ u = 0 $.
\end{proof}
\end{prop}
\begin{proof}[Proof of Theorem \ref{thm 1}]
According to  Remark \ref{remark 1}, the operator $ P_{1} $ has a closure $ \overline{P_{1}} $ from $ C^{\infty}_0 (\mathbb{T}^{3}\times \mathbb{R}^3 ) $. This gives Theorem \ref{thm 1}, by a direct application of Hille-Yosida's theorem (cf.\ \cite{pazy2012semigroups} for more details for the semi-group theory) to the problem \eqref{2}, with $ D ( P_{1}) = C^{\infty}_0 (\mathbb{T}^{3}\times \mathbb{R}^3 )$ and $ H = L^{2} (dx d \mu) $.
\end{proof}
From now on, we write $ P_{\theta} $ for the closure of the operator   $ P_{\theta} $  from the space $ C^{\infty}_0 (\mathbb{T}^{3} \times \mathbb{R}^3) $ with respect the norm $L^2(\mathbb{T}^{3}\times \mathbb{R}^{3})$.
\section{Trend to the equilibrium}
\subsection{Hypocoercivity in the space $ L^{2} (dxd \mu) $ }
The purpose of this subsection is to show the exponential time decay of the $ L^{2} (dxd \mu) $ entropy for $ P_1 $, based on macroscopic equations.
First, we try to find the macroscopic equations associated with system \eqref{2}. We write  $ f $ in the following form:
\begin{align}
f(x,v) = r(x) + h(x,v) ,
\label{19}
\end{align} 
where \quad $r(f)(x)= \displaystyle\int f(x,v)\, d\mu (v)$  and   $m(f)(x)= \displaystyle\int v f(x,v) \, d\mu (v)$ will be use later.
\begin{defi}
In the following, we define
$$\Lambda_{x}=(1-\Delta_{x})^{1/2}, $$
and introduce a class of Hilbert spaces
$$\mathbb{H}^{\alpha}:=\{u\in \mathcal{S}^{'}, \Lambda_{x}^{\alpha}\in L^{2}(dxd\mu)  \}\quad \text{ with } \alpha \in \mathbb{R},$$
where $\mathcal{S}^{'}$ is the space of temperate distributions.

We recall that the operator $ \Lambda_{x}^2 $ is an elliptic, self-adjoint, invertible operator from $\mathbb{H}^{2}(dxd\mu)$ to $L^2 (dxd\mu)$ and $\Lambda_{x}\geq  Id$. (cf \cite[section 6]{herau2004isotropic} for a proof of these properties).
\end{defi}
\begin{lem}\label{lem 1}
Let $f $ be the solution of the system \eqref{2}, with the decomposition given in  \eqref{19}. Then we have
\begin{align}
& \partial_{t} r = Op_{1}(h), \\
& \partial_{t}m = - \nabla_{x} r -m \wedge B_{e} + Op_{1}(h).
\label{20}
\end{align}
Where $ Op_{1} $ denotes a bounded generic operator of $ L^{2} $ to $ \mathbb{H}^{- 1} $.
\end{lem}
\begin{proof}
We suppose is $f$ is a Schwarz function. In order to show  equation $ (14) $, we integrate  equation \eqref{2} with respect to the measure $d\mu$ . We get
\begin{align*}
\partial_{t} \int f d\mu + \int v \cdot \nabla_{x} f d\mu & - \int (v \wedge B_{e}) \cdot \nabla_{v}f d\mu  = - \int (-\nabla_{v} +v )\cdot\nabla_{v} f d\mu \\ &=\langle Lf, 1 \rangle = \langle f, L 1\rangle = 0,
\end{align*}
since, $ L 1 = 0 $ ,  $ L $ is a self-adjoint operator and 
\[ (v \wedge B_{e}) \cdot \nabla_{v}f = \nabla_{v}\cdot (v \wedge B_{e} )f.\]
Then, we obtain 
\begin{align*}
\partial_{t} r = \nabla_{x} \cdot \int v h d\mu = Op_{1}(h),
\end{align*}
hence  equality $ (14). $ \\
To show $ (15) $, we multiply  equation \eqref{2} by $v$ before performing an integration with respect to the measure $d\mu$, we obteinning
\begin{align}\label{eg 20}
\partial_{t} \int v f d\mu + \nabla_{x}\cdot \int  (v\otimes v)f\, d\mu - \int v ((v \wedge B_{e}) \cdot \nabla_{v}f ) d\mu = \langle Lf, v \rangle ,
\end{align}
where $\nabla_{x}\int (v\otimes v) f\,d\mu =\sum_{i=1}^{3}\sum_{j=1}^{3}\,\int\, v_{j}v_{i}\partial_{x_{i}}f\,d\mu $. 
Now, we will calculate term by term  the left-hand side of the equality \eqref{eg 20}. We first observe that
$$\nabla_{x}\cdot \int\,(v\otimes v)f\,d\mu= Op_{1}(h)+\nabla_{x}r.$$
Furethermore,
\begin{align*}
\langle Lf, v \rangle = \langle  f, Lv \rangle = \langle f,v \rangle = \int h vd\mu.
\end{align*}
It remains to compute component by component  $ \displaystyle\int v \left( (v\wedge B_{e} \cdot \nabla_{v} f )\right) d\mu  $. We have for all $ 1\leq j \leq 3$,
\begin{align*}
\int v_{j} ((v \wedge B_{e} )\cdot \nabla_{v} f ) d\mu &= \int v_{j} \nabla_{v}\cdot ((v\wedge B_{e})f) d\mu \\
& = \int (-\nabla_{v}+ v)(v_{j} )\cdot (v\wedge B_{e} ) f d\mu \\
& = - \delta_{j} \cdot\int (v\wedge B_{e})f d\mu \\
& = - \delta_{j} \cdot \left( \int v f d\mu \right)\wedge B_{e}  \\
& = - \delta_{j}\cdot (m \wedge B_{e}) \\
& = - (m\wedge B_{e})_{j}.
\end{align*}
Therefore  $\displaystyle\int v \left( (v\wedge B_{e} \cdot \nabla_{v} f )\right) d\mu = - m \wedge B_{e}$, where $m$ is defined in \eqref{19}.\\
By combining all the previous equalities in \eqref{eg 20}, we obtain
\begin{align*}
 \partial_{t}m = - \nabla_{x}r - m \wedge B_{e} + Op_{1}(h).
\end{align*}
\end{proof}
\begin{remarque}
Under Hypothesis \ref{hyp 1}, since $ B_{e} \in  L^{\infty} (\mathbb{T}^{3}) $,
$$ m \wedge B_{e} = Op_{1} (h), $$
so the macroscopic equation \eqref{19} takes the following form:
\begin{align*}
\partial_{t} m = - \nabla_{x} r + Op_{1} (h).
\end{align*}
\end{remarque}
Now we are ready to build a new entropy, defined for any $u\in L^{2}(dxd\mu)$ by
$$ \mathcal{F}_{\varepsilon}(u)= \Vert u \Vert^{2} + \varepsilon \langle \Lambda_{x}^{-2} \nabla_{x}r(u), m(u) \rangle, \quad r(u):=\int \, u\,d\mu \text{ and } m(u):=\int vu\,d\mu .$$
Using the Cauchy-Schwarz inequality gives us  directly that 
\begin{lem}\label{lem 2}
If $ \varepsilon \leq \frac{1}{2}$, then 
\begin{align}\label{ineg 16}
 \frac{1}{2} \Vert u \Vert^{2}  \leq \mathcal{ F}_{\varepsilon}(u) \leq  2 \Vert u\Vert^{2} 
 \end{align}
 \end{lem}
Now, we can prove the main result of hypocoercivity leading to the proof of Theorem \ref{thm 2}.
\begin{prop} \label{prop 3}
There exists $ \kappa > 0 $ such that, if $ f_{0} \in L^{2} (dx d \mu) $ and $ \langle f_{0} \rangle = 0 $,
then the solution of  system \eqref{2} satisfies
$$ \forall t \geq 0, \quad \mathcal{F}_{\varepsilon}(f(t)) \leq e^{-\kappa t} \mathcal{F}_{\varepsilon}(f_{0}).$$
\begin{proof}
We write 
$$ \frac{d}{dt} \mathcal{F}_{\varepsilon}(f(t)) = \frac{d}{dt} \Vert f \Vert^{2} + \varepsilon \frac{d}{dt} \langle \Lambda_{x}^{-2}r, m \rangle .$$
We will omit the dependence of $ f $ with respect to $ t $. For the first term, we notice that
\begin{align}\label{ineg 17}
\frac{d}{dt} \Vert f \Vert^{2} & = 2 \langle L f, f  \rangle 
 = - 2 \Vert \nabla_{v} f \Vert^{2} \leq -2 \Vert h  \Vert^{2},
\end{align} 
by the spectral property of the operator $ L $. For the second term, using the macroscopic equations, we get
\begin{align*}
\frac{d}{dt} \langle \Lambda_{x}^{-2} \nabla_{x}r, m \rangle & = \langle \Lambda_{x}^{-2}\nabla_{x}\partial_{t} r, m \rangle + \langle \Lambda_{x}^{-2} \nabla_{x} r, \partial_{t} m \rangle \\
& = -\langle \Lambda^{-2}_{x} \nabla_{x} r, \nabla_{x} r \rangle +  \langle \Lambda_{x}^{-2}\nabla_{x} Op_{1}(h), m \rangle + \langle \Lambda_{x}^{-2}\nabla_{x}r, Op_{1}(h) \rangle\\
& \leq \Vert \Lambda_{x}^{-1}\nabla_{x} r \Vert^{2} + C \Vert \Lambda_{x}^{-1} Op_{1}(h) \Vert \left( \Vert \Lambda_{x}^{-1}\nabla_{x} r \Vert + \Vert \Lambda_{x}^{-1}\nabla_{x} m \Vert \right) .
\end{align*}
Now, using $\Vert m \Vert \leq \Vert h \Vert$, the  Cauchy-Schwarz inequality and the following estimate:
$$ \Vert \Lambda_{x}^{-1} \nabla_{x} \phi \Vert \leq \Vert \phi \Vert, \quad \forall \phi \in L^{2}(dx d\mu),$$
we obtain
$$ \frac{d}{dt}\langle \Lambda_{x}^{-2}r, m \rangle \leq - \frac{1}{2}  \Vert \Lambda_{x}^{-1} \nabla_{x} r \Vert^{2} + C \Vert h \Vert^{2} .$$
 Poincar\'e's inequality on $ L^{2}(dx) $  takes the form
$$ \forall \phi \in L^{2}(dx),\quad \text{such that}\quad \langle \phi \rangle =0, \quad \Vert \Lambda_{x}^{-1} \nabla_{x} \phi \Vert^{2} \geq \frac{c_{P}}{c_{P} + 1} \Vert \phi \Vert^{2}, $$
where $\langle \phi\rangle =\int \phi (x)\,dx$ and $ c_{P} >0$ is the spectral gap of $ - \Delta_{x} $ on the torus (see \cite[Lemma 2.6]{herau2017introduction} for the proof of the previous inequality). Using this, we obtain, by applying the previous estimate to $ r $ ( since $\langle r \rangle = \langle  f \rangle = \langle f_{0}  \rangle =0 $),
\begin{align}\label{ineg 18}
\frac{d}{dt} \langle \Lambda_{x}^{-2} \nabla_{x}r, m \rangle \leq -\frac{1}{2} \frac{c_{P}}{c_{P} + 1} \Vert r \Vert^{2} + C \Vert h \Vert^{2}.
\end{align}
gathering \eqref{ineg 17} and \eqref{ineg 18}, we get
$$ \frac{d}{dt} \mathcal{F}_{\varepsilon}(f) \leq - \Vert h \Vert^{2} -\frac{\varepsilon}{2} \frac{c_{P}}{c_{P} + 1} \Vert r \Vert^{2} + C\varepsilon \Vert h \Vert^{2}$$
Now we choose $ \varepsilon $ such that $ C \varepsilon \leq \frac{1}{2} $, we get
\begin{align*}
\frac{d}{dt} \mathcal{F}_{ \varepsilon}(f) & \leq   -\frac{1}{2} \Vert h \Vert^{2} -\frac{\varepsilon}{2} \frac{c_{P}}{c_{P} + 1} \Vert r \Vert^{2} \leq -\frac{\varepsilon}{2} \frac{c_{P}}{c_{P} + 1} \Vert f \Vert^{2} \leq -\frac{\varepsilon}{4} \frac{c_{P}}{c_{P} + 1} \mathcal{F}_{\varepsilon}(f). 
\end{align*}
Which gives the result with $\kappa = \frac{\varepsilon}{4} \frac{c_{P}}{c_{P} + 1}>0$.
\end{proof}
\end{prop}
We can deduce the proof of Theorem \ref{thm 2}.
\begin{proof}[Proof of Theorem \ref{thm 2}]
Starting from  Lemma \ref{lem 2} and Proposition \ref{prop 3}, we have, for $ f $  the solution of the system \eqref{2},
$$ \Vert f \Vert^{2} \leq 2 \mathcal{F}_{\varepsilon}(f) \leq 2 e^{-\kappa t} \mathcal{F}_{\varepsilon}(f_{0}) \leq  4 e^{- \kappa t} \Vert f_{0} \Vert^{2}. $$
This completes the proof of Theorem \ref{thm 2}.
\end{proof}

\subsection{Hypocoercivity in the space $H^{1}(dxd\mu)$ }
We will establish some technical lemmas, which will help us to deduce the exponential time decay of the norm $ H^{1}(dxd\mu) $, noting that we work in $ 3 $ dimensions.
~\par 
The following lemma gives the exact values of some commutators will be used later.
\begin{lem} \label{lem 3} The following equalities
\begin{enumerate}
\item $[ \partial_{v_{i}}, v\cdot \nabla_{x} ] = \partial_{x_i} \quad \forall i\in\{1,2,3\}.$ 
\item $[ \partial_{v_i}, (-\partial_{v_j} + v_j)] = \delta_{ij} \quad \forall i,j\in \{ 1,2,3\}.$
\item $[ \nabla_{v}, (v\wedge B_{e})\cdot \nabla_{v} ] = B_{e} \wedge \nabla_{v}$.
\item $[\nabla_{x}, (v\wedge B_{e})\cdot \nabla_{v} ] = (v\wedge \nabla_{x}B_{e})\cdot \nabla_{v}.$
\end{enumerate}
\begin{proof}
Let $f\in C^{\infty}_{0} ( \mathbb{T}^{3} \times\mathbb{R}^{3})$. 
The first two equalities are obvious. We  directly go to the proof of $(3)$ in component. Writting $ B_{e} = (B_{1}, B_{2}, B_{3}) $,
\begin{align*}
[ \partial_{v_{1}}, (v\wedge B_{e})\cdot \nabla_{v} ]f & = \partial_{v_{1}} ((v\wedge B_{e})\cdot \nabla_{v})f - ((v\wedge B_{e})\cdot \nabla_{v})\partial_{v_{1}}f \\
& = \partial_{v_{1}} [  ( v_{2} B_{3} - v_{3} B_{2})\partial_{v_{1}}f + (v_{3} B_{1} - v_{1} B_{3})\partial_{v_{2}}f \\ 
& + (v_{1} B_{2} - v_{2} B_{1})\partial_{v_{3}}f )]
 -  ((v\wedge B_{e})\cdot \nabla_{v})\partial_{1}f \\
 &= (B_{2} \partial_{v_{3}}f - B_{3}\partial_{v_{2}} f ) \\
 &= (B_{e} \wedge \nabla_{v})_{1} f.
\end{align*}
Similarly we can show that, for all  $ 1\leq i \leq 3 $,
$$ [ \partial_{v_{i}}, (v\wedge B_{e})\cdot \nabla_{v} ] =  (B_{e} \wedge \nabla_{v} )_{i} . $$
This proves the equality $(3)$. 
Now, we will show $(4)$,
\begin{align*}
[\nabla_{x}, (v\wedge B_{e})\cdot \nabla_{v} ]f & = \nabla_{x} ((v \wedge B_{e})\cdot \nabla_{v} )f - ((v\wedge B_{e})\cdot\nabla_{v})\nabla_{x} f \\
& =( v \wedge \nabla_{x}B_{e})\cdot \nabla_{v}f.
\end{align*}
\end{proof}
\end{lem}

 Now, we are ready to build a new entropy that will allow us to show the exponential decay of the norm $ H^{1} (dx d  \mu) $.  
We define this modified entropy by
 $$ \mathcal{E}(u) = C \Vert u \Vert^{2} + D \Vert \nabla_{v} u \Vert^{2} + E \langle \nabla_{x}u, \nabla_{v}u \rangle + \Vert \nabla_{x}u \Vert^{2},\quad \forall u\in H^{1}(dxd\mu),$$
 where $ C> D> E> 1 $ are constants fixed below.  We first show that $ \mathcal {E} (u) $ is equivalent to the norm $ H^{1} (dxd \mu) $ of $ u $.
 \begin{lem}\label{lem 4}
  If $E^{2} < D$, then $\forall u \in H^{1}(dxd\mu)$ 
 \begin{align}\label{ ing 14}
  \frac{1}{2} \Vert u \Vert^{2}_{H^{1}(dxd\mu)} \leq \mathcal{E}(u) \leq 2C \Vert u \Vert^{2}_{H^{1}(dxd\mu)}. 
   \end{align}
  \begin{proof}
 Let $u\in H^{1}(dx d\mu)$. Using the Cauchy-Schwarz inequality, we get
 \begin{align*}
\vert E\langle \nabla_{x}u, \nabla_{v}u \rangle  \vert \leq \frac{ E^{2}}{2} \Vert \nabla_{v} u \Vert^{2} + \frac{1}{2} \Vert \nabla_{x} u \Vert^{2}, 
 \end{align*}
 which implies,
 \begin{align*}
 C \Vert u \Vert^{2} + (D - \frac{ E^{2}}{2} )  \Vert \nabla_{v} u \Vert^{2} + (1 - \frac{1}{2}) & \Vert \nabla_{x} u \Vert^{2} \leq\mathcal{E}(u) \\
 & \leq  C \Vert u \Vert^{2} + (D + \frac{ E^{2}}{2} )  \Vert \nabla_{v} u \Vert^{2} + (1 + \frac{1}{2})  \Vert \nabla_{x} u \Vert^{2}.
 \end{align*}
 This implies \eqref{ ing 14} if $E^{2} < D$.
 \end{proof}
 \end{lem}
Note that using the same approach as in Section $ 3 $, we can show the existence of a solution of the problem \eqref{2}, which will be denoted as $ f $, in the space $ H^{1} (dxd \mu) $ in the sense of an associated semi-group. Using the preceding results, we are able to study the decrease of the modified entropy $ \mathcal{E} (f (t)) $.
  \begin{prop}\label{prop 4}
  Suppose that $ B_{e} $ satisfies the Hypothesis \ref{hyp 1}, then there exist $ C, D, E $ and $ \kappa> 0 $, such that for all $ f_{0} \in H^{1} (dxd \mu) $ with $ \langle f_{0} \rangle = 0 $, the solution $ f $ of the system \eqref{2} satisfies 
   $$ \forall t >0, \quad \mathcal{E}(f(t)) \leq \mathcal{E}(f_{0})\, e^{-\kappa t} .$$
   \begin{proof}
  The time derivatives of the four terms defining $ \mathcal{E}(f (t)) $ will be calculated separately. For the first term we have
   \begin{align*}
 \frac{d}{dt} \Vert f \Vert^{2} & = - 2\langle \partial_{t}f, f \rangle \\
 & = -2 \underbrace{\langle v\cdot \nabla_{x}f, f \rangle}_{=0} + 2 \underbrace{\langle (v\wedge B_{e})\cdot \nabla_{v}f, f \rangle}_{=0} - 2\langle (-\nabla_{v} + v)\cdot\nabla_{v}f, f  \rangle \\
 & = -2 \Vert \nabla_{v} f \Vert^{2}.
 \end{align*}
  The second term writes
   \begin{align*}
 \frac{d}{dt} \Vert \nabla_{v}f \Vert^{2} & = 2 \langle \nabla_{v} \nabla_{t}f, \nabla_{v} f \rangle \\
 & = -2 \langle \nabla_{v}(v \cdot \nabla_{x}f), \nabla_{v}f\rangle + 2 \langle \nabla_{v}((v\wedge B_{e})\cdot \nabla_{v}f), \nabla_{v}f \rangle \\
 &-2 \langle \nabla_{v} (-\nabla_{v} + v )\cdot \nabla_{v}f, \nabla_{v}f \rangle
\\ & = -2 \underbrace{\langle v \nabla_{x} \nabla_{v} f, \nabla_{v} f \rangle}_{=0} -2 \langle  [\nabla_{v}, v\cdot \nabla_{x}]f, \nabla_{v}f  \rangle + 2 \langle [\nabla_{v}, (v\wedge B_{e})\cdot \nabla_{v}] f, \nabla_{v}f \rangle \\
& + 2 \underbrace{\langle ((v\wedge B_{e})\cdot \nabla_{v})\nabla_{v} f, \nabla_{v} f \rangle}_{=0} - 2 \Vert (-\nabla_{v} +v )\cdot \nabla_{v} f \Vert^{2}.
 \end{align*}
 We used the fact that the operators  $ v \cdot \nabla_{x} $ and $ (v \wedge B_{e}) \cdot \nabla_{v} $ are skew-adjoint in $L^{2}(dxd\mu)$ by Lemma \ref{lem 14}. According to equalities $(1)$ and $(3)$ of Lemma \ref{lem 3}, we then obtain
 \begin{align*}
\frac{d}{dt} \Vert \nabla_{v} f \Vert^{2} = -2 \langle \nabla_{x} f,\nabla_{v}f \rangle + 2 \langle (B_{e} \wedge \nabla_{v})f,\nabla_{v}f \rangle -2 \Vert (-\nabla_{v} + v)\cdot\nabla_{v} f \Vert^{2}.
\end{align*}
The time derivative of the third term can be calculated as follows:
\begin{align}\label{ineg 20}
\frac{d}{dt} \langle \partial_{v}f, \nabla_{x} f \rangle = \langle \nabla_{v}\partial_{t}f, \nabla_{x}f \rangle + \langle \nabla_{v}f,  \nabla_{x} \partial_{t}f\rangle.
\end{align}
We calculate each term of equality \eqref{ineg 20}. For the first term,  using equalities $(1), (2)$ and $(3)$ of Lemma \ref{lem 3}, we obtain
\begin{align*}
\langle \nabla_{v}\partial_{t} f, \nabla_{x}f  \rangle &= - \langle \nabla_{v}(v\cdot \nabla_{x}f - (v\wedge B_{e} )\cdot \nabla_{v} f + (-\nabla_{v} + v)\cdot\nabla_{v} f ), \nabla_{x} f \rangle \\
&= -\Vert  \nabla_{x} f \Vert^{2} -\langle v\cdot\nabla_{x}(\nabla_{v}f), \nabla_{x}f \rangle - \langle \nabla_{v} f, \nabla_{x}f \rangle - \langle \Delta_{v} f, \nabla_{v}\cdot \nabla_{x}f \rangle \\
& + \langle ( B_{e} \wedge\nabla_{v})f, \nabla_{x} f \rangle + \langle ((v\wedge B_{e}) \cdot \nabla_{v})\nabla_{v}f, \nabla_{x}f \rangle.
\end{align*} 

For the second term of equality \eqref{ineg 20}, using  equality $(4)$ of Lemma \ref{lem 3}, we have
\begin{align*}
\langle \nabla_{v}f, \nabla_{x}\partial_{t}f \rangle &= -\langle \nabla_{v}f, \nabla_{x}(v\cdot \nabla_{x}f - (v\wedge B_{e})\cdot \nabla_{v}f + (-\nabla_{v} + v)\cdot\nabla_{v} f) \rangle \\ 
&= -\langle \nabla_{v}f, v\cdot\nabla_{x} (\nabla_x f) \rangle + \langle \nabla_{v}f, (v\wedge \nabla_{x}B_{e})\cdot \nabla_{v}f \rangle \\
&+ \langle \nabla_{v}f, ((v\wedge B_{e})\cdot \nabla_{v}) \nabla_{x} f \rangle -\langle \nabla_{v} f, \nabla_{x}f \rangle - \langle \Delta_{v} f, \nabla_{v}\cdot \nabla_{x}f \rangle.
\end{align*} 
 Combining the proceding equalities of the two terms in \eqref{ineg 20}, we get
\begin{align*}
\frac{d}{dt} \langle \nabla_{v}f, \nabla_{x}f \rangle =& -\Vert \nabla_{x} f \Vert^{2} - \langle\nabla_{v}f, \nabla_{x}f \rangle + 2\langle (-\nabla_{v} + v )\cdot\nabla_{v}f, \nabla_{v}\cdot(\nabla_{x}f)  \rangle \\
& - [\langle v\cdot\nabla_{x}(\nabla_{v} f), \nabla_{x}f \rangle + \langle  \nabla_{v}f, v   \cdot \nabla_{x}(\nabla_x f) \rangle] \\
& + \langle (B_{e} \wedge\nabla_{v})f, \nabla_{x}f \rangle + \langle \nabla_{v}f, (v\wedge \nabla_{x}B_{e})\cdot \nabla_{v}f \rangle \\
& + [\langle ((v\wedge B_{e}) \cdot \nabla_{v})\nabla_{v}f, \nabla_{x}f \rangle + \langle \nabla_{v}f, ((v\wedge B_{e})\cdot \nabla_{v}) \nabla_{x} f \rangle].
\end{align*} 
According to Lemma \ref{lem 14}, the operators $ v \cdot \nabla_{x} $ and $ (v \wedge B_{e}) \cdot \nabla_ {v} $ are skew-adjoint in $ L^{2} (dxd \mu) $, we have
\begin{align}
&\langle v\cdot\nabla_{x}(\nabla_{v} f), \nabla_{x}f \rangle + \langle  \nabla_{v}f, v   \cdot \nabla_{x}(\nabla_x f) \rangle = 0 \label{eg 26}\\
&\langle ((v\wedge B_{e}) \cdot \nabla_{v})\nabla_{v}f, \nabla_{x}f \rangle + \langle \nabla_{v}f, ((v\wedge B_{e})\cdot \nabla_{v}) \nabla_{x} f \rangle =0\label{eg 27}.
\end{align}
Using equality \eqref{eg 26}-\eqref{eg 27}, we obtain
\begin{align*}
\frac{d}{dt} \langle \nabla_{v}f, \nabla_{x}f \rangle &= -\Vert \nabla_{x} f \Vert^{2} - \langle \nabla_{v}f, \nabla_{x}f \rangle + 2\langle (-\nabla_{v} + v )\nabla_{v}f, \nabla_{v}\cdot(\nabla_{x}f ) \rangle\\
& + \langle (B_{e} \wedge \nabla_{v})f, \nabla_{x}f \rangle + \langle \nabla_{v}f, (v\wedge \nabla_{x}B_{e})\cdot \nabla_{v}f \rangle .
\end{align*}
Eventually, the time derivative of the last term takes the following form
\begin{align*}
\frac{d}{dt} \Vert\nabla_{x} f \Vert^{2}&= 2\langle \nabla_{x}\partial_{t}f, \nabla_{x}f \rangle\\
&= -2 \underbrace{\langle \nabla_{x}(v\cdot \nabla_{x}f), \nabla_{x}f \rangle}_{=0} + 2 \langle \nabla_{x}((v\wedge B_{e})\cdot \nabla_{v}f), \nabla_{x}f \rangle \\&-2 \langle \nabla_{x} (-\nabla_{v} + v )\cdot \nabla_{v}f, \nabla_{x}f \rangle \\
&= -2 \Vert \nabla_{x}\nabla_{v}f \Vert^{2} + 2 \langle (v\wedge \nabla_{x}B_{e})\cdot \nabla_{v}f, \nabla_{x}f \rangle + 2\underbrace{\langle ((v\wedge B_{e})\cdot \nabla_{v})\nabla_{x}f, \nabla_{x}f \rangle }_{=0} \\
&= -2 \Vert \nabla_{x}\nabla_{v}f \Vert^{2} + 2 \langle (v\wedge \nabla_{x}B_{e})\cdot \nabla_{v}f, \nabla_{x}f \rangle,
\end{align*}
By collecting all the ties, we get
\begin{align*}
\frac{d}{dt} \mathcal{E}(f) = &-2C \Vert \nabla_{v}f \Vert^{2} -2D \Vert (-\nabla_{v} + v)\cdot\nabla_{v} f \Vert^{2} - E \Vert \nabla_{x} f \Vert^{2} - 2 \Vert \nabla_{x}\nabla_{v} f \Vert^{2} \\
& -(2D + E) \langle \nabla_{x}f, \nabla_{v}f \rangle + 2E \langle (-\nabla_{v} +v)\cdot\nabla_{v} f, \nabla_{v}\cdot(\nabla_{x} f )\rangle \\
& + 2D\langle (B_{e} \wedge \nabla_{v})f, \nabla_{v}f \rangle + E \langle (B_{e} \wedge \nabla_{v})f, \nabla_{x}f \rangle \\
&+ E \langle (v\wedge \nabla_{x}B_{e})\cdot\nabla_{v} f, \nabla_{v}f \rangle + 2 \langle (v\wedge \nabla_{x}B_{e})\cdot\nabla_{v}f, \nabla_{x}f \rangle.
\end{align*}
Now, we need the following technical lemma.
\begin{lem}\label{lem 5}
We have the following equalities in $L^2 (dxd\mu)$:
\begin{enumerate}
\item[i.] $\langle (v\wedge \nabla_{x}B_{e})\cdot\nabla_{v} f, \nabla_{v}f \rangle = -\langle \nabla_{v} \wedge (\nabla_{x}B_{e}\cdot \nabla_{v}f), \nabla_{v}f \rangle. $
\item[ii.] $\langle (v\wedge \nabla_{x}B_{e})\cdot\nabla_{v}f, \nabla_{x}f \rangle =-\langle \nabla_{v} \wedge (\nabla_{x}B_{e}\cdot \nabla_{x}f), \nabla_{v}f \rangle. $
\end{enumerate}
where $ \nabla_{x} B_{e} $ is the Jacobian matrix of the function $$ x \to B_{e}(x)=(B_{1}(x),B_{2}(x),B_{3}(x)),$$
and 
\begin{align*}
\langle (v\wedge \nabla_{x}B_{e})\cdot \nabla_{v}f, \nabla_{v} f \rangle  &=   \sum_{i=1}^{3} \left[ \int \partial_{v_{1}}\mu (v) \left( (\nabla_{x}B_{e})_{i2} \, \partial_{v_{3}}f - (\nabla_{x}B_{e})_{i3}\,\partial_{v_{2}}f \right) \partial_{v_{i}}f \, dx dv\right] \\ 
& +  \sum_{i=1}^{3} \left[ \int \partial_{v_{2}}\mu (v) \left( (\nabla_{x}B_{e})_{i3} \, \partial_{v_{1}}f - (\nabla_{x}B_{e})_{i1}\,\partial_{v_{3}}f \right) \partial_{v_{i}}f \, dx dv \right] \\
&+ \sum_{i=1}^{3} \left[ \int \partial_{v_{3}}\mu (v) \left( (\nabla_{x}B_{e})_{i1} \, \partial_{v_{2}}f - (\nabla_{x}B_{e})_{i2}\,\partial_{v_{1}} f \right) \partial_{v_{i}}f \, dx dv \right] .
\end{align*}
\begin{proof}
Using the fact that $v \,\mu(v) = \nabla_{v}(\mu(v))$ and  integrations by part, we obtain the result by simple computations.
\end{proof}
\end{lem}
Let's go back to the proof of Proposition \ref{prop 4}. Using  Lemma \ref{lem 5}, the time derivative of $ \mathcal{E} (f (t)) $ takes the following form:
\begin{align}
\frac{d}{dt} \mathcal{E}(f(t)) = &-2C \Vert \nabla_{v}f \Vert^{2} -2D \Vert (-\nabla_{v} + v)\cdot\nabla_{v} f \Vert^{2} - E \Vert \nabla_{x} f \Vert^{2} - 2 \Vert \nabla_{x}\nabla_{v} f \Vert^{2} \notag \\
& -(2D + E) \langle \nabla_{x}f, \nabla_{v}f \rangle + 2E \langle (-\nabla_{v} +v)\cdot\nabla_{v} f, \nabla_{v}\cdot(\nabla_{x} f) \rangle \notag \\
& + 2D\langle (B_{e} \wedge \nabla_{v})f, \nabla_{v}f \rangle + E \langle (B_{e} \wedge \nabla_{v})f,\nabla_{x}f \rangle \notag \\
&-E\langle \nabla_{v} \wedge (\nabla_{x}B_{e}\cdot \nabla_{v}f), \nabla_{v}f \rangle - 2 \langle \nabla_{v} \wedge (\nabla_{x}B_{e}\cdot \nabla_{x}f), \nabla_{v}f \rangle . \notag
\end{align}
Now we estimate the scalar products in the previous equality in $ L^{2} (dxd \mu) $. For all $  \eta, \eta^{'} $ and $ \eta^{''}> 0 $, we have
\begin{align*}
&\vert (2D+E) \langle \nabla_{x}f, \nabla_{v}f \rangle \vert \leq \frac{1}{2} \Vert \nabla_{x}f  \Vert^{2} + \frac{1}{2} (2D+E)^{2} \Vert \nabla_{v}f \Vert^{2},
\\ &\vert 2E \langle (-\nabla_{v} + v)\cdot\nabla_{v} f, \nabla_{v}\cdot(\nabla_{x}f )\rangle \vert  \leq \Vert \nabla_{v}\cdot( \nabla_{x} f) \Vert^{2} + E^{2} \Vert (-\nabla_{v} + v )\cdot\nabla_{v} f \Vert^{2}, \\
& \vert 2D\langle (B_{e} \wedge \nabla_{v})f, \nabla_{v}f  \rangle \vert \leq 2D \Vert B_{e} \Vert_{\infty}\, \Vert \nabla_{v} f \Vert^{2} ,\\
& \vert E\langle (B_{e} \wedge \nabla_{v})f, \nabla_{x}f \rangle \vert \leq E \frac{\eta}{2} \Vert \nabla_{x}f \Vert^{2} + E\frac{1}{2\eta} \Vert B_{e} \Vert_{\infty}^{2} \Vert \nabla_{v}f \Vert^{2}, 
\end{align*}
and using than $\Vert \nabla^{2}_{v}f \Vert^{2} \leq \Vert (-\nabla_{v} + v)\cdot\nabla_{v} f \Vert^{2} + \Vert \nabla_{v} f \Vert^{2}$, we obtain
\begin{align*}
\vert E\langle (\nabla_{v} \wedge (\nabla_{x}B_{e} \cdot \nabla_{v}))f, \nabla_{v}f  \rangle \vert &\leq \frac{\eta^{'}}{2} \Vert \nabla_{v}^{2} \, f\Vert^{2} + \frac{E^{2}}{2 \eta^{'}} \Vert \nabla_{x} B_{e} \Vert_{\infty}^{2} \,\Vert \nabla_{v}f \Vert^{2} \\
&\leq \frac{\eta^{'}}{2} \Vert (-\nabla_{v} + v )\cdot\nabla_{v} f\Vert^{2} + (\frac{E^{2}}{2 \eta^{'}} \Vert \nabla_{x} B_{e} \Vert_{\infty}^{2}  \\& + \frac{\eta^{'}}{2} ) \Vert \nabla_{v}f \Vert^{2}, 
\end{align*}
The last scalar product is bounded by

\begin{align*}
\vert 2\langle (\nabla_{v} \wedge (\nabla_{x}B_{e} \cdot \nabla_{x}))f, \nabla_{v}f \rangle \Vert \leq \eta^{''} \Vert \nabla_{x} \nabla_{v} f \Vert^{2}  + \frac{1}{\eta^{''}} \Vert \nabla_{x} B_{e}\Vert_{\infty}^{2} \, \Vert \nabla_{v} f \Vert^{2}.
\end{align*}

Combining all the previous estimates, we have
\begin{align*}
\frac{d}{dt} \mathcal{E}(f) &\leq ( -2C + \frac{1}{2} (2D+E)^{2} + 2D \Vert B_{e} \Vert_{\infty} + \frac{E}{2\eta} \Vert B_{e} \Vert^{2}_{\infty} \\
& +\frac{E^{2}}{2 \eta^{'}} \Vert \nabla_{x} B_{e} \Vert_{\infty}^{2}   + \frac{\eta^{'}}{2} + \frac{1}{\eta^{''}}\Vert \nabla_{x} B_{e}\Vert_{\infty}^{2} ) \, \Vert \nabla_{v} f
 \Vert^{2}\\
 &+ (-2D + E^{2}+ \frac{\eta^{'}}{2})\, \Vert (-\nabla_{v} + v)\cdot\nabla_{v} f \Vert^{2} \\
 &+ (-E + \frac{1}{2} + \frac{\eta}{2}E) \, \Vert\nabla_{x} f \Vert^{2} \\
 &+ (-2 + 1 + \eta^{''}) \, \Vert \nabla_{x}\nabla_{v} f \Vert^{2}.
\end{align*}
We notice that
\begin{align*}
A &= \frac{1}{2} (2D+E)^{2} + 2D \Vert B_{e} \Vert_{\infty} + \frac{E}{2\eta} \Vert B_{e} \Vert^{2}_{\infty} \\
& +\frac{E^{2}}{2 \eta^{'}} \Vert \nabla_{x} B_{e} \Vert_{\infty}^{2}   + \frac{\eta^{'}}{2} + \frac{1}{\eta^{''}}\Vert \nabla_{x} B_{e}\Vert_{\infty}^{2} .
\end{align*}
We choose $ \eta $, $ \eta^{''} $, $ E $, $ D $ and $ C $ such that
\begin{enumerate}
\item $ \eta \leq 1 $ and $\eta^{''} \leq 1$ .
\item $E \geq 2  $.
\item $D \geq  \frac{1}{2} (E^{2} + \frac{\eta^{'}}{2}).$
\item $C \geq A $.
\end{enumerate}
Under the previous conditions, we get
\begin{align*}
\frac{d}{dt} \mathcal{E}(f) & \leq - C \Vert \nabla_{v}f \Vert^{2} - \frac{E}{4} \Vert \nabla_{x} f \Vert^{2} 
 \leq  - \frac{E}{4} (\Vert \nabla_{v} f \Vert^{2} + \Vert \nabla_{x} f \Vert^{2}).
\end{align*}
Using the Poincar\'e inequality in space and velocity variables, we then obtain
\begin{align*}
\frac{d}{dt} \mathcal{E}(f) \leq - \frac{E}{8} (\Vert \nabla_{v} f \Vert^{2} + \Vert \nabla_{x} f \Vert^{2}) -\frac{E}{8}c_{p} \Vert f \Vert^{2} \leq  -\frac{E}{8}  \frac{c_{p}}{2C} \mathcal{E}(f).
\end{align*}
Which completes Proposition \ref{prop 3} with $\displaystyle\kappa = \frac{E}{8}  \frac{c_{p}}{2C}>0$.
   \end{proof}
  \end{prop}
   \begin{proof}[Proof of Theorem \ref{thm 4}.]
   Using Lemma \ref{lem 4} and Proposition \ref{prop 4}, we get $ \kappa> 0 $ and $ 1 <E <D <C $ such that 
   \begin{align*}
 \Vert f \Vert_{H^{1}(dxd\mu)}^{2} \leq 2\, \mathcal{E}(f) &\leq 2C e^{-\kappa t} \mathcal{E}(f_{0}) \\
 &\leq 4C e^{-\kappa t}\Vert f_{0} \Vert_{H^{1}(dxd\mu)}^{2}.
 \end{align*}
 This completes the proof of Theorem \ref{thm 4}.
 \end{proof}

 \section{Enlargement of the functional space}\label{section 4}
 \subsection{Intermediate results}\label{subs 4.1}
 In this section, we  extend the results of exponential time decay of the semi-group to  enlarged spaces (which we will define later), following the recent work of Gualdani, Mischler, Mouhot in \cite{gualdani2010factorization}.\\\\
 \textbf{Notation:} Let $ E $ be a Banach space. 
\begin{enumerate}
\item[-] We denote by $ \mathcal{C} (E) $  the space of unbounded, closed operators with dense domains in $ E $.
\item[-] We denote by $B(E)$ the space of  bounded operators in $E$.
\item[-] Let $ a \in \mathbb{R} $. We define the complex half-plane 
$$ \Delta_{a} = \{z \in \mathbb{C}, \mathrm{Re}\, z> a \}. $$
\item[-] Let $ L \in \mathcal{C} (E) $. $ \Sigma (L) $ denote the spectrum of the operator $ L $ and $\sigma_{d}(L)$ its discrete spectrum.
\item[-] Let $ \xi \in \Sigma_{d}(L) $, for $r$ sufficiently small we define the spectral projection associated with $ \xi $ by
$$ \Pi_{L, \xi}: = \frac{1}{2i \pi} \int_{\vert z- \xi \vert = r}\,(L-z)^{-1} dz. $$
  \item[-] Let $ a \in \mathbb{R} $ be such that $ \Delta_{a} \cap \Sigma (L) = \{\xi_{1}, \xi_{2}, ..., \xi_{k} \} \subset \Sigma_{d} (L) $. We define $ \Pi_{L, a} $ as the operator
$$ \Pi_{L, a} = \sum\limits_{j = 1}^{k} \, \Pi_{L, \xi_{j}}. $$
\end{enumerate}
We need the following definition on the convolution of semigroup (corresponding to composition at the level of the resolvent operators).
\begin{defi}[Convolution of time dependent operators]
Let $X_1 ,
X_2$ and $X_3$ be Banach spaces. For two given functions
$$ \mathcal{S}_1 \in L^1 (\mathbb{R}^{+} ; B(X_1 , X_2 )) \text{ and } \mathcal{S}_2 \in L^1 (\mathbb{R}^{+} ; B(X_2 , X_3 )),$$
we define the convolution $\mathcal{S}_2 * \mathcal{S}_1 \in L^1 (\mathbb{R}^{+} ; B(X_1 , X_3 ))$ by
$$ \forall t \geq 0, (\mathcal{S}_2 * \mathcal{S}_1 )(t) :=
\int_{0}^{t}\, \mathcal{S}_2 (s) \mathcal{S}_1 (t -s) \,ds.$$
\end{defi}
When $\mathcal{S} = \mathcal{S}_1 = \mathcal{S}_2$ and $X_1 = X_2 = X_3$ , we define inductively $\mathcal{S}^{(*1)} = \mathcal{S} \text{ and }
\mathcal{S}^{(*\ell)} = \mathcal{S} * \mathcal{S}^{(*(\ell -1))} \text{ for any } \ell\geq  2.$

 We say that $ L\in \mathcal{C}(E)$ is hypodissipative if it is dissipative for some norm equivalent to the canonical norm of $ E$ and we say that $L$ is dissipative for the norm $\Vert \cdot\Vert_{E} $ on $E$ if
 $$\forall f \in D(L ), \forall f^{*} \in  E^{*} \text{ such that } \langle f, f^{*} \rangle  = \Vert f\Vert_{E} = \Vert f^{*} \Vert_{E^{*}} ,
\mathrm{Re}\, \langle Lf, f^{*} \rangle \leq 0.$$
 We refer to the paper \cite[Section 2.3]{gualdani2010factorization} for an introduction to this subject. Now, we recall  the crucial Theorem of enlargement of the functional space.
 \begin{thm}[Theorem 2.13 in \cite{gualdani2010factorization}]\label{thm 9}
Let  $ E $ and $ \mathcal{E} $ be two Banach spaces such that $ E \subset \mathcal {E} $,
 $ L \in \mathcal{C} (E) $ and $ \mathcal{L} \in \mathcal{C} (\mathcal{E}) $ such that $ \mathcal{L}_{|_{E} } = L $. We suppose that there exist $\mathcal{A}$ and $\mathcal{B} \in \mathcal{C} (\mathcal{E})$ such that  
 $ \mathcal{L} = \mathcal{A} + \mathcal{B} $ (with corresponding
restrictions $A, B$  on $E$). Suppose there exists $ a \in \mathbb{R} $ and  $ n \in \mathbb{N} $ such that
 \begin{itemize}
 \item[$(H_{1})$] \textbf{Locating the spectrum of $ L $:} \\
 $$ \Sigma (L) \cap \Delta_{a} = \{0 \} \subset \Sigma_{d} (L), \quad $$
 and $ L-a $ is dissipative on $ \mathrm{Im} (Id_{E} - \Pi_{L,0}) $
 \item[$(H_{2})$] \textbf{Dissipativity of $ \mathcal{B} $ and bounded character of $ \mathcal{A} $:}
 $ (\mathcal{B}-a) $ is hypodissipative on $ \mathcal{E} $ and $ \mathcal{A} \in B (\mathcal{E}) $ and $ A \in B (E) $.
 \item[$(H_{3})$] \textbf{Regularization properties of $ T_{n} (t) = \left(\mathcal{A} S_{\mathcal{B}} (t) \right)^ {( * n)} :$}
 $$ \Vert T_{n} (t) \Vert_{B (\mathcal{E}, E)} \leq C_{a, n} \, e^{at}. $$
\end{itemize}

Then for  all $ a'> a $, we have the following estimate:
$$ \forall t\geq 0, \quad \Vert S_{\mathcal{L}} (t) -S_{\mathcal{L}} (t) \Pi_{\mathcal{L}, 0} \Vert_{B(\mathcal{E})} \leq C_{a'}\,e^{a' t} . $$
\end{thm}
To finish this subsection, we give a lemma  providing a practical criterion to prove hypothesis $ (H_3) $ in the previous theorem.
\begin{lem}[Lemma $2.4$ in \cite{mischler2016exponential}]\label{lem 8}
Let $E$ and $\mathcal{E}$ be two Banach spaces with $E \subset \mathcal{E}$  dense with continuous embedding, and consider $L  \in \mathcal{C}(E)$ and $\mathcal{L} \in \mathcal{C}(\mathcal{E})$ with $\mathcal{L}_{\mid E}=L$ and $a \in \mathbb{R}$. Let us assume that:
\begin{enumerate}
\item[a)] $\mathcal{B}-a$ is hypodissipative on $\mathcal{E}$ and $B-a$ on $E$.
\item[b)] $\mathcal{A} \in B(\mathcal{E})$ and $ A \in B(E)$.
\item[c)] There are constants $b\in \mathbb{R}$ and $\Theta \geq 0$ such that  
$$\Vert S_{\mathcal{B}}(t)\mathcal{A}\Vert_{B(\mathcal{E},E)} \leq Ce^{bt}\,t^{-\Theta} \text{ et } \Vert \mathcal{A} S_{\mathcal{B}}(t)\Vert_{B(\mathcal{E},E)} \leq Ce^{bt}\,t^{-\Theta}. $$
\end{enumerate}
Then for all $a'>a$, there exist some explicit constants  $n \in \mathbb{N}$ and $C_{a'}\geq 1$, such that
$$\forall t\geq 0, \quad \Vert T_{n}\Vert_{B(\mathcal{E},E)}\leq C_{a'}\, e^{a't}.$$
\end{lem}
\subsection{ Study of the magnetic-Fokker-Planck operator on the spaces $ L ^ {p} (m) $ and $\tilde{W}^{1,p}(m)$:}\label{subs 4.2}
This part consists in building the general framework of the problem.
 ~\par Recall first the equation of Fokker-Planck \eqref{1}  written in original variable:
 \begin{align}\label{eq 36}
\partial_{t} F &= -P_{0} \, F, \quad   F(0,x,v)=  F_{0}(x,v),\\
\text{where}\quad & -P_{0}\, F=\nabla_{v}\cdot (\nabla_{v} F + K \, F) - v\cdot \nabla_{x}F,\notag
\end{align} 
and where we recall that $ P_0 $ was introduced in Section 2 and with $$ K (x,v) = v + v \wedge B_{e} (x) = \nabla_{v} \Phi + U ,\text{ where }   \Phi (v) = \frac{\vert v \vert^{2}}{2}$$
and $ B_{e} $ is the external magnetic field  satisfying  Hypothesis \ref{hyp 1}. As mentioned in Section $2$, the Maxwellian $ \mu $ is  a solution of the system \eqref{1}. We will need the following modified Poincar\'e inequality:
\begin{align}
\iint_{\mathbb{T}^{3}\times\mathbb{R}^{3}}\, \displaystyle \left\vert \nabla_{v} \left(\frac{F}{\mu}\right)\right\vert^{2}&\,\mu(v)dxdv \notag
\\&\geq 2 \lambda_{p}\, \iint_{\mathbb{T}^{3}\times\mathbb{R}^{3}} \left( F - \int_{\mathbb{R}^{3}}\,F(v')\,dv'\right)^{2}\,(1+\vert \nabla_{v}\Phi \vert^{2})\,\mu^{-1}(v)dxdv , 
\label{10} 
\end{align}
where $ \lambda_{p}> 0 $ which depends on the dimension (see \cite[Lemma 3.6]{mischler2016exponential}). See also \cite{mouhot2006quantitative}, \cite{bakry2008simple} and \cite{ane2000inegalites}. 
~\par Now we will define define the expanded functional space.
\begin{defi}
  Let $ m = m (v)> 0 $ on $ \mathbb{R}^{3} $  be a weight of class $ C^{\infty} $ and recall that
\begin{itemize}
\item[$\bullet$] The space $ L^{p} (m) $ for $ p \in [1,2] $, is the Lebesgue space with weight associated with the norm
\begin{align*}
\Vert F \Vert_{L^{p}(m)} := \Vert Fm\Vert_{L^{p}} = \left( \int_{\mathbb{R}^{3}\times \mathbb{T}^{3}}\, F^{p}(x,v)\,m^{p}(v) dv dx\right)^{\frac{1}{p}}.
\end{align*}
 \item[$\bullet$] We define the technical function $ \Psi_{m, p} $ by
$$\Psi_{m,p} := (p-1)\,\frac{\vert \nabla_{v} m \vert^{2}}{m^{2}} + \frac{\Delta_{v} m}{m} +  (1-\frac{1}{p} )\, \nabla_{v} \cdot K+ K\cdot \frac{\nabla_{v} m }{m} ,$$
where $K(x,v)=v\wedge B_e (x) +v$.
\end{itemize}
\end{defi}
We will show the decay of the semi-group associated with the problem \eqref{1} in the spaces $ L^{p} (m) $ where $ p \in [1 , 2] $, when $ m $ verifies the following hypothesis:
\begin{enumerate}
\item[$\mathrm{(W_p)}$] The weight  $m$ satisfies $L^{2}(\mu^{-\frac{1}{2}}) \subset L^{p}(m)$ with continuous injection and  \[\label{defi 15} \limsup_{\vert v \vert \rightarrow +\infty } \Psi_{m,p}  := a_{m,p} <0 .\]
\end{enumerate}
\begin{remarque}\label{remarque 2}
 In the following, we note $ m_{0} = \mu^{- 1/2} $ the exponential weight. By direct computation, 
  $L^2 (\mu^{-1/2} ) \subset
L^{q }(m_0 )$ for any $q \in [1, 2]$ with continuous injection and there exists $b\in \mathbb{R}$ such that 
\begin{equation}
\begin{cases}
\sup\limits_{q\in [1,2], v \in \mathbb{R}^{3}} \, \Psi_{m_{0},q} \leq b \\
\sup\limits_{v \in \mathbb{R}^{3}} \left( \displaystyle \frac{\Delta_{v} m_{0}}{m_{0}} - \frac{\vert\nabla_{v}m_{0}\vert^{2}}{m_{0^{2}}}\right) \leq b.
\end{cases}
\end{equation}
(See Lemma $3.7$ in \cite{gualdani2010factorization} for a proof of the previous property). Under the previous hypothesis, by direct computation we obtain that the semi-group  $\mathcal{S}_{L_0}$ is bounded from $L^p (m_0)$ to $L^p (m_0).$
\end{remarque}
~\par We work now in $ L^{p} (m) $  with a polynomial weight $ m $ satisfying Hypothesis $ \mathrm{(W_p)} $.
\begin{lem}
Let $ m = \langle v \rangle^{k}:=(1+\vert v\vert^2)^{k/2}  $ and $p\in [1,2]$. Then hypothesis $\mathrm{(W_p)}$ is true when $ k $ satisfies the following estimate:
 \[  k > 3(1-\displaystyle \frac{1}{p}).\]
 \begin{proof}
 For the proof, see  Lemma $3.7$ in \cite{gualdani2010factorization}.
 \end{proof}
\end{lem}
\subsubsection{Proof of Theorem \ref{thm 5}}
From now on, we write $ L_0 $ for the operator $ -P_0 $, the Fokker- Planck operator considered on the space $ L^2 (m_{0}) $ defined in \eqref{eq 36} (respectively $ \mathcal{L}_{0} $ for $-P_0$ the Fokker-Planck operator considered on the space $ \mathcal{E} = L^{p} (m) $, with $ m = \langle v \rangle^{k} $, where $ k> 3 (1- \displaystyle \frac{1}{p }) $ and $ p \in [1,2] $) .
We will prove Theorem \ref{thm 5} by applying Theorem \ref{thm 9} to $ \mathcal{L}_0 $. To verify Hypotheses $ (H_ {2}) $ and $ (H_3) $ of Theorem \ref{thm 9}, we  need two lemmas about the dissipativity and regularization properties of $\mathcal{L}_{0}$ following \cite{gualdani2010factorization}.
\begin{defi}\label{defi 1}
 We split operator $\mathcal{L}_{0}$ into two pieces: for $M$,  $R>1$, we define the operator $\mathcal{B}$ by 
\begin{align}
 \mathcal{B} = \mathcal{L}_{0}-\mathcal{A}\, \text{ with } \, \mathcal{A}f= M\chi_{R}f, 
 \label{def *}
 \end{align}
 where $\chi_{R}(v) = \chi (v/R)$, and $0 \leq \chi \in C_{0}^{\infty} (\mathbb{T}^{3} \times \mathbb{R}^{3})$ is such that  $\chi (v)  = 1 \, \text{ when } \, \vert v \vert \leq 1. $ 
 We also denote by $A$ and $B$ the restriction of the operators $\mathcal{A}$ and $\mathcal{B}$ to the space $E$.
 \end{defi}
\begin{lem}[Dissipativity of $\mathcal{B}$]\label{lem 10}
Under Assumption $\mathrm{(W_p)}$, for all  $ 0>a>a_{m,p}$, we can choose $R, M>1$ such that the operator $ \mathcal{B}-a$ satisfies the dissipativity estimate for some $C>0$
$$ \forall t \geq 0, \quad \Vert S_{\mathcal{B}} (t)f \Vert_{L^{p}(m)} \leq Ce^{at} \, \Vert f \Vert_{L^{p}(m)} .  $$
\begin{proof}
The proof follows the one given in Lemma $3.8$ in \cite{gualdani2010factorization}.
Let $F$ be smooth, rapidly decaying and positive function $F$. Since of $ \Psi_{m, p}$ is independent of the magnetic field, 
 by integration by parts with respect to $v$ and using  Remark \ref{remarque A.2}, we have
\begin{align*}
\displaystyle\frac{1}{p}\frac{d}{dt}\, \Vert F\Vert_{L^{p}(m)}^{p} &= \iint_{ \mathbb{T}^{3}\times \mathbb{R}^{3}}\, \left( \mathcal{L}_{0}F - M\chi_{R}(v) F \right)\,  \vert F \vert^{p-2} \, F \,m^{p}(v) \, dxdv \\
&= -(p-1)\iint_{ \mathbb{T}^{3}\times \mathbb{R}^{3}} \,\vert  \nabla_{v} F\vert^{2} \vert F \vert^{p-2}m^{p}(v)\, dxdv\\
&+ \iint_{ \mathbb{T}^{3}\times \mathbb{R}^{3}} \, \vert F \vert^{p} \Psi_{m,p} \,m^{p}(v) \, dxdv - \iint_{ \mathbb{T}^{3}\times \mathbb{R}^{3}} \, M\, \chi_{R}(v) \vert F\vert^{p}\,m^{p}(v) \, dxdv\\ 
& \leq \iint_{ \mathbb{T}^{3}\times \mathbb{R}^{3}}\, \vert F \vert^{p} (\Psi_{m,p} - M\chi_{R})\, m^{p}(v) \,dxdv.
\end{align*}
Let now take $a>a_{m,p}$. As $ m $ satisfies the hypothesis $ (W_p)$,  there exist $ M $ and $ R $ two  large constants such that
$$ \forall v\in \mathbb{R}^{3},  \quad \Psi_{m,p} - M\chi_{R} \leq a, $$
 and we obtain
$$  \displaystyle\frac{1}{p}\frac{d}{dt}\, \Vert F\Vert_{L^{p}(m)}^{p} \leq a\, \int_{ \mathbb{T}^{3}\times \mathbb{R}^{3}} \vert F \vert^{p}\, m^{p}(v) \, dxdv.$$
This completes the proof of Lemma \ref{lem 10}.
\end{proof}
\end{lem} 

From now on, $a$, $M$ and $R$ are fixed. 
We note that $ \mathcal{B}^{*} $ is the dual operator of $ \mathcal{B} $ relative to the pivot space $ L^{2}(\mathbb{T}^3\times\mathbb{R}^{3}) $, which is defined as follows:
 $$\mathcal{B}^{*}F:=\nabla_{v}\cdot (\nabla_{v}F -K\, F)+v\cdot\nabla_{x}F-M\chi_R F. $$ 
\begin{lem}[Regularization properties]\label{lem 11}
There exists $b\in \mathbb{R}$ and $C > 0$ such that, for all $t\geq 0$, 
\begin{align*}
 \forall 1\leq p\leq q \leq 2,\quad &\Vert S_{\mathcal{B}}(t) F_{0}\Vert_{L^{q}(m_{0})} \leq C\, e^{bt} \, t^{- (3d+1)(\frac{1}{p} - \frac{1}{q})}\,  \Vert F_{0}\Vert_{L^{p}(m_{0})},\\
\forall 2\leq q'\leq p'\leq +\infty, \quad & \Vert S_{\mathcal{B}^{*}}(t) F_{0}\Vert_{L^{p'}(m_{0})} \leq C\, e^{bt} \, t^{- (3d+1)(\frac{1}{p} - \frac{1}{q})}\,  \Vert F_{0}\Vert_{L^{q'}(m_{0})},
\end{align*}
 where $ p' $ and $ q' $ are the conjugates of $ p $ and $ q $ respectively
\begin{proof}
We consider $ F (t) $ the solution of the evolution equation
$$\partial_{t}F(t)=\mathcal{B}F(t),\quad F{|_{t=0}} = F_{0}. $$
We introduce the following entropy defined for all
 $t\in [0,T]$, with $T\ll 1$ and $r>1$ to be fixed later:
$$\mathcal{H}(t,h) = B\Vert h\Vert_{L^{1}(m_{0})}^{2}+t^{r}\mathcal{G}(t,h), $$
with $$\mathcal{G}(t,h)=\alpha\Vert h \Vert^{2}_{L^{2}(m_{0})} + D\, t \Vert \nabla_{v} h \Vert^{2}_{L^{2}(m_{0})}  + E \, t^{2} \langle \nabla_{x}h, \nabla_{v}h \rangle_{L^2 (m_0)} + \beta t^{3}\, \Vert \nabla_{x} h\Vert^{2}_{L^{2}(m_{0})},  $$
where $B>\alpha>D, \beta $, $E<\sqrt{\beta D}$ and $r$ is an integer that will be determined later. We will omit the dependence of $F$ on $t$. 
Using the methods and computations of the proof of Proposition \ref{prop 4} and  adapting the techniques used in \cite{herau2017introduction}, we choose the constants $ \alpha, D $ and $ E> 0 $ large enough such that there exist a constant $ C_{\mathcal{G}}> 0 $ (depending on $\Vert B_{e}\Vert_{L^{\infty}(\mathbb{T}^{3})}$ and $\Vert \nabla_{x}B_e\Vert_{L^{\infty}(\mathbb{T}^{3})}$) such that
\begin{align*}
\frac{d}{dt}\mathcal{G}(t,F)&\leq -C_{\mathcal{G}}(\Vert \nabla_{v}F\Vert_{L^{2}(m_{0})}^{2}+t^{2}\Vert \nabla_{x}F\Vert_{L^{2}(m_{0})}^{2})\\
 & +\left(\frac{M}{2}\Vert \Delta_{v}\chi_{R}\Vert_{L^{\infty}(\mathbb{T}^{3})}^{2}+\frac{M}{2}\Vert \chi_{R}\Vert_{L^{\infty}(\mathbb{T}^{3})}^{2}+M\Vert \nabla_{v}\chi_{R}\Vert_{L^{\infty}(\mathbb{T}^{3})}^{2}\right)\Vert F\Vert_{L^{2}(m_{0})}^{2}\\
&\leq -C_{\mathcal{G}}(\Vert \nabla_{v}F\Vert_{L^{2}(m_{0})}^{2}+t^{2}\Vert \nabla_{x}F\Vert_{L^{2}(m_{0})}^{2})+C_{\chi} \Vert F\Vert_{L^{2}(m_{0})}^{2}.
\end{align*}
Here, $ C_{\chi}> 0 $ is a uniform constant in $ R >1$ but depends on $ M $.  
\begin{align*}
 \frac{d}{dt}\mathcal{H}(t,F)&=B\displaystyle\frac{d}{dt}\Vert F\Vert_{L^{1}(m_{0})}^{2}+r\,t^{r-1}\mathcal{G}(t,F)+t^{r}\displaystyle\frac{d}{dt}\mathcal{G}(t,F)\\
 &\leq B\displaystyle\frac{d}{dt}\Vert F\Vert_{L^{1}(m_{0})}^{2}+r\,t^{r-1}\mathcal{G}(t,F)\\
 &-C_{\mathcal{G}}\,t^{r}(\Vert \nabla_{v}F\Vert_{L^{2}(m_{0})}^{2}+t^{2}\Vert \nabla_{x}F\Vert_{L^{2}(m_{0})}^{2})+C_{\chi}\,t^{r}\Vert F\Vert_{L^{2}(m_{0})}^{2}.
\end{align*}
We choose the constants $ \beta $ and $ T> 0 $ such that $$ \beta<\frac{C_{\mathcal{G}}}{2r} \text{ and }
 T\leq  \frac{C_{\mathcal{G}}}{2r}(\frac{1}{D}+\frac{1}{\beta})$$ .

We deduce that
\begin{align}
\frac{d}{dt}\mathcal{H}(t,F)&\leq B\displaystyle\frac{d}{dt}\Vert F\Vert_{L^{1}(m_{0})}^{2}-\frac{C_{\mathcal{G}}}{2}t^{r}\left( \Vert \nabla_{v}F\Vert_{L^{2}(m_{0})}^{2}+t^{2}\Vert \nabla_{x}F\Vert_{L^{2}(m_{0})}^{2}\right)\notag\\
&+\frac{C_{\chi}}{2}t^{r-1}\,\Vert F\Vert_{L^{2}(m_{0})}^{2}.
\label{ineg pr 1}
\end{align}
Now, the Nash inequality \cite{nash1958continuity} implies that there exists $ C_{d}> 0 $ such that
\begin{align}
\iint_{\mathbb{T}^{d}\times \mathbb{R}^{d}}\,\vert F(x,v)\vert^{2}\,m_{0}^{2}\,dxdv \leq &C_{d}\left( \iint_{\mathbb{T}^{d}\times \mathbb{R}^{d}}\, \vert \nabla_{x,v}(F\,m_{0})\vert^{2}\,dxdv\right)^{\frac{d}{d+1}}\notag\\
&\times\left( \iint_{\mathbb{T}^{d}\times \mathbb{R}^{d}}\,\vert F\vert m_{0}\,dxdv\right)^{\frac{2}{d+1}}.
\label{nash 1}
\end{align}
We need to have an estimate based on $ \Vert\nabla_{x, v} F \Vert_ {L^{2} (m_{0})} $. Firstly,
\begin{align}
\iint \vert \nabla_{v} (F m_0)\vert^{2}\,dxdv &=\iint \vert \nabla_{v}F +\frac{v}{2}\,F\vert^{2}\, m_0^{2}\,dxdv\notag \\
&\leq 2\left( \iint \vert \nabla_{v}F\vert^{2} \, m_0^2\,dxdv + \iint \vert F\vert^{2}\vert v\vert^{2}\,m_0^2\,dxdv\right)\notag\\
&\leq 2\left(\Vert \nabla_{v}F\Vert_{L^{2}(m_{0})}^{2} + \Vert v\,F\Vert_{L^{2}(m_{0})}^{2}\right).
\label{deri 1}
\end{align}
On the other hand, we use  the fact that $v\,m_{0}^{2}=\nabla_{v} (m_{0}^{2})$ to estimate $\Vert vF\Vert_{L^{2}(m_{0})}$. We get
\begin{align*}
\iint \vert F\vert^{2}\vert v\vert^{2}\,m_0^2\,dxdv &= \iint v\,\vert F\vert^{2}\cdot v\,m_0^{2}\,dxdv \\
&=\iint v\,\vert F\vert^{2}\cdot \nabla_{v}(m_{0}^{2})\,dxdv,
\end{align*}
and integrating by parts in  $ v $ in the previous estimate, we obtain
\begin{align*}
\iint \vert F\vert^{2}\vert v\vert^{2}\,m_0^2\,dxdv & \leq -\iint \nabla_{v}\cdot (v\,\vert F\vert^{2})\,m_{0}^2 \,dxdv\\
&=-3\iint \vert F\vert^{2}\,m_{0}^2\,dxdv -2\iint v\cdot
F\,\nabla_{v}F\, m_{0}^2\,dxdv\\
&\leq -2\iint v\cdot F\nabla_{v}F\,m_{0}^{2}\,dxdv.
\end{align*}
Applying Cauchy-Schwarz inequality, we get
\begin{align*}
\iint \vert F\vert^{2}\vert v\vert^{2}\,m_0^2\,dxdv &\leq 2 \left( \iint \vert v\vert^{2}\vert F\vert^{2} \,m_{0}^{2}\,dxdv\right)^{1/2} \times \left( \iint \vert \nabla_{v} F\vert^{2} \,m_{0}^{2}\,dxdv\right)^{1/2}\\
&\leq 8\,\iint \vert \nabla_{v} F\vert^{2} \,m_{0}^{2}\,dxdv +\frac{1}{2}\,\iint \vert v\vert^{2}\vert F\vert^{2} \,m_{0}^{2}\,dxdv.
\end{align*}
Therefore
\begin{align}
\iint \vert F\vert^{2}\vert v\vert^{2}\,m_0^2\,dxdv \leq 16 \,\iint \vert \nabla_{v} F\vert^{2} \,m_{0}^{2}\,dxdv.
\label{deri 2}
\end{align}
Using the previous estimate and inequality \eqref{deri 1}, we have
\begin{align*}
\iint \vert \nabla_{v}(Fm_0)\vert^2 \, dxdv \leq 34\, \iint \vert \nabla_{v}F\vert^{2}\,m_{0}^2\,dxdv.
\end{align*}
Using the previous inequality and the fact that $ \nabla_{x} (Fm_{0}) = m_{0} \, \nabla_{x} F $ (since $ m_0 $ does not depend on $ x $),  there exists $ C^{'}_d> 0 $ such that the estimate \eqref{nash 1} becomes
\begin{align*}
\iint_{\mathbb{T}^{d}\times \mathbb{R}^{d}}\,\vert F(x,v)\vert^{2}\,m_{0}^{2}\,dxdv \leq &C^{'}_{d}\left( \iint_{\mathbb{T}^{d}\times \mathbb{R}^{d}}\, \vert \nabla_{x,v}(F)\vert^{2}\,m_{0}^{2}\,dxdv\right)^{\frac{d}{d+1}}\notag\\
&\times\left( \iint_{\mathbb{T}^{d}\times \mathbb{R}^{d}}\,\vert F\vert m_{0}\,dxdv\right)^{\frac{2}{d+1}}.
\end{align*}
Using Young's inequality with $p=(d+1) \text{ and } q=(d+1)/d$ , we get, for all $ \varepsilon >0$, 
\begin{align*}
\iint_{\mathbb{T}^{d}\times \mathbb{R}^{d}}\,\vert F(x,v)\vert^{2}\,m_{0}^{2}\,dxdv &\leq C^{'}_{d}
 t^{-3d/d+1}\left( \iint_{\mathbb{T}^{d}\times \mathbb{R}^{d}}\,\vert F\vert m_{0}\,dxdv\right)^{\frac{2}{d+1}}\\
&\times t^{3d/d+1}\left( \iint_{\mathbb{T}^{d}\times \mathbb{R}^{d}}\, \vert \nabla_{x,v}F\,\vert^{2} m_{0}^2\,dxdv\right)^{\frac{d}{d+1}}\\
&\leq C_{\varepsilon,d}\,t^{-3d}\,\Vert F\Vert_{L^{1}(m_{0})}^{2}+\varepsilon \,t^{3}\,\Vert\nabla_{x,v} F\Vert_{L^{2}(m_{0})}^{2}.
\end{align*}
Using the previous estimate, we choose $ \varepsilon > 0 $ small enough that there is a $ C^{''}> 0 $
$$\frac{d}{dt}\mathcal{H}(t, F) \leq B\dfrac{d}{dt}\Vert F\Vert_{L^{1}(m_{0})}^{2}+C^{''}\,t^{r-1-3d}\,\Vert F\Vert_{L^{1}(m_{0})}^{2}.$$
According to Remark \ref{remarque 2} there exists $ b \in \mathbb{R} $ such that $\forall p\in [1,2]$
$$\frac{d}{dt} \Vert F \Vert_{L^{p}(m_{0})} \leq b \Vert F \Vert_{L^{p}(m_{0})},\quad\forall t\geq 0 ,$$
Finally, using the previous estimate when $ p = 1 $ and choosing $ r = 3d + 1 $, we deduce that there exists $ B^{''}> 0 $ such 
$$\frac{d}{dt}\mathcal{H}(t, F) \leq B^{''}\Vert F\Vert_{L^{1}(m_{0})}^{2}\leq \frac{B^{''}}{B}\, \mathcal{H}(t,F).$$
Thanks to Gronwall's Lemma,  there exists $ B''' > 0 $ such that
 $$\forall t\in [0,T], \quad\mathcal{H}(t,F)\leq  B'''\,\mathcal{H}(0,F_{0})\leq C\Vert F_0\Vert^{2}_{L^1 (m_0)}.  $$
Then, 
 $$\forall t\in (0,T], \quad\Vert F \Vert_{L^{2}(m_{0})}^{2} \leq \frac{\alpha}{t^{r}}\,\mathcal{H}(t,F)\leq \frac{C}{t^{3d+1}}\,\Vert F_0\Vert_{L^{1}(m_{0})}^{2}.  $$
As a consequence, using the continuity  of $ S_{\mathcal{B}} (t) $ on $ L^p (m_0) $ with $ p = 2 $,
\begin{align*}
\forall t\in (T,+\infty), \quad \Vert F\Vert_{L^{2}(m_{0})}^{2}=\Vert \mathcal{S}_{\mathcal{B}}(t-T+T)F_{0}\Vert_{L^2(m_{0})}^{2}& \leq 
C\,e^{(t-T)b}\,\Vert \mathcal{S}_{\mathcal{B}}(T) F_{0}\Vert^{2}_{L^2 (m_0)},
 \end{align*}
 and eventually for all $t\in (0,+\infty)$
 \begin{align*}
 \Vert F\Vert_{L^{2}(m_{0})}^{2} \leq \frac{C}{t^{3d+1}}\,\Vert F_{0}\Vert_{L^{1}(m_{0})}^{2}
 \end{align*}
~\par Let us now consider $p$ and $q$ satisfying  $1\leq p\leq q\leq 2$. $S_{\mathcal{B}}(t)$ is continuous from $L^{p} (m_0)$ into $L^q (m_0)$ using the Riesz-Thorin Interpolation Theorem. Moreover, if we denote by $C_{p,q}(t) $ the norm of $S_{\mathcal{B}}(t):L^p (m_0) \to L^q (m_0)$, we get the following estimate:
$$C_{p,q}(t) \leq C_{2,2}^{2-\frac{2}{p}} (t)\, C_{1,1}^{\frac{2}{q}-1}(t)\, C_{1,2}^{\frac{2}{p}-\frac{2}{q}}(t) \leq C\frac{e^{bt}}{t^{(3d+1)(1/p -1/q)}}.$$
This shows the first estimate.
~ \par  Now we will show the second estimate. According to the first estimate, we have
$$\forall 1\leq p\leq q \leq 2,\quad \Vert S_{\mathcal{B}}(t) F_{0}\Vert_{L^{q}(m_{0})} \leq C\, e^{bt} \, t^{- (3d+1)(\frac{1}{p} - \frac{1}{q})}\,  \Vert F_{0}\Vert_{L^{p}(m_{0})}, $$
which means
$$\Vert S_{m_0 \, \mathcal{B} \, m_0^{-1}} (t)h \Vert_{L^{q}}  \leq   C\, e^{bt} \, t^{- (3d+1)(\frac{1}{p} - \frac{1}{q})}\, \Vert h\Vert_{L^{p}},$$
where $h=m_0 F_0$. Then by duality, we get
 $$ \Vert S_{m_0 \, \mathcal{B}^{*} \, m_0^{-1}} (t)h \Vert_{L^{p^{'}}}  \leq   C\, e^{bt} \, t^{- (3d+1)(\frac{1}{p} - \frac{1}{q})}\, \Vert h\Vert_{L^{q^{'}}},$$ 
 where $ p^{'} $ and $ q^{'} $ are the conjugates of $ p $ and $ q $ respectively.
 Which gives the result by reusing the definition of weighted dual spaces
 $$ \Vert S_{\mathcal{B}^{*}}(t) F_{0}\Vert_{L^{p^{'}}(m_{0})} \leq C\, e^{bt} \, t^{- (3d+1)(\frac{1}{p} - \frac{1}{q})}\,  \Vert F_{0}\Vert_{L^{q^{'}}(m_{0})}.$$
 This completes the proof.
 \end{proof}
\end{lem}

\begin{cor}\label{cor 1}
Let $ m $ be a weight that satisfies Hypothesis \ref{hyp 2}, then there exists $ \Theta \geq 0 $ such that for all $  F_{0} \in L^p (m) $ with $ p \in [1,2] $, we have the following estimate
\begin{align*}
&\forall t\geq 0, \quad\Vert \mathcal{A}S_{\mathcal{B}(t)}F_{0}\Vert_{L^{2}(m_{0})} \leq Ce^{bt}\,t^{-\Theta}\,\Vert F_{0}\Vert_{L^{p}(m)},\\
&\forall t\geq 0, \quad \Vert S_{\mathcal{B}(t)}\mathcal{A}F_{0}\Vert_{L^{2}(m_{0})} \leq Ce^{bt}\,t^{-\Theta}\,\Vert F_{0}\Vert_{L^{p}(m)}.
\end{align*}
\begin{proof} We first prove the second inequality.
Let $ F_{0} \in L^p (m) $ with $ m $ a polynomial weight satisfying Hypothesis \ref{hyp 2}. For all $ 1 \leq p \leq 2 $ and for all $  t \in ] 0,1] $ and $ \ v \in \mathbb{R}^{3} $, using Lemma \ref{lem 11} with $q=2$, we get
\begin{align*}
\Vert S_{\mathcal{B}}(t)\mathcal{A}F_{0}\Vert_{L^{2}(m_{0})}&\leq Ce^{bt}\,t^{-(3d+1)(\frac{1}{p}-\frac{1}{2})}\Vert \mathcal{A}F_0\Vert_{L^p(m_{0})} \\
&\leq Ce^{bt}\,t^{-(3d+1)(\frac{1}{p}-\frac{1}{2})}\left\Vert \mathcal{A}F_0\times \dfrac{m_{0}}{m}\right\Vert_{L^{p}(m)}\\
&\leq C\,M\,e^{bt}\,t^{-(3d+1)(\frac{1}{p}-\frac{1}{2})}\times \left(\sup\limits_{v\in B(0,R)}\,\frac{m_{0}(v)}{m(v)}\right)\, \Vert F_{0}\Vert_{L^p (m)}\\
&\leq C^{'}\,e^{bt}\,t^{-(3d+1)(\frac{1}{p}-\frac{1}{2})}\,\Vert F_{0}\Vert_{L^{p}(m)}\leq C'\,e^{bt}t^{-\Theta} \,\Vert F_{0}\Vert_{L^p(m)},
\end{align*}
where $\Theta=(3d+1)(1/p -1/2)>0$.
 ~\par To show the first estimate, we proceed step by step.\\
 \textbf {Step 1: } First, we will show the following estimate:
 \begin{align}
\label{duality 1}
\Vert S_{\mathcal{B}^{*}}(t)  g \Vert_{L^{p'}(m)} \leq C\,e^{bt}\, t^{-\Theta}\,\Vert g\Vert_{L^{2}(m_0)},\quad\forall t\geq 0.
\end{align}
Indeed, using the continuous and dense injection $ L^{p'} (m_0) \subset L^{p'} (m) $, we obtain
$$\Vert S_{\mathcal{B}^{*}}(t)  g\Vert_{L^{p'}(m)} \leq \Vert S_{\mathcal{B}^{*}}(t)\,g\Vert_{L^{p'}(m_0)} ,$$
then using Lemma \ref{lem 11} with $ q' = 2 $, we obtain
\begin{align}
\label{estimation 6}
\Vert S_{\mathcal{B}^{*}}(t)\,g\Vert_{L^{p'}(m_0)} \leq C\, e^{bt}\, t^{-\Theta}\,\Vert g\Vert_{L^{2}(m_0)}, \quad \forall t\geq 0,
\end{align}
where $\Theta =(3d+1)(1/p-1/2)$.\\
\textbf{Step 2:} Of the inequality \eqref{estimation 6}, it follows that for $ g = \mathcal{A} F_{0} $, we get
\begin{align*}
\Vert S_{\mathcal{B}^{*}}\, \mathcal{A} \,F_0 \Vert_{L^{p'}(m)}\leq C\,e^{bt}\, t^{-\Theta}\,\Vert \mathcal{A}F_{0}\Vert_{L^{2}(m_0)},
\end{align*}
which means, denoting $ h = m F_0 $
\begin{align*}
\Vert S_{m\mathcal{B}^{*}\,m^{-1}}\, \mathcal{A} \,h \Vert_{L^{p'}}\leq C\,e^{bt}\, t^{-\Theta}\,\left\Vert \mathcal{A} h\times \frac{m_0}{m}\right\Vert_{L^{2}}\leq C' \,e^{bt}\,t^{-\Theta}\Vert h\Vert_{L^{2}},
\end{align*}
by a duality argument and noting that $ \mathcal{A}^{*} = \mathcal{A} $, we get
\begin{align*}
\Vert \mathcal{A}\,S_{m\mathcal{B}\,m^{-1}}  \,h \Vert_{L^{2}}\leq C\,e^{bt}\, t^{-\Theta}\,\left\Vert  h \right\Vert_{L^{p}}.
\end{align*}
Finally, according to our definition of weighted dual spaces  and replacing $ h $ by $mF_0$,  we obtain
\begin{align}
\Vert \mathcal{A}\,S_{\mathcal{B}}(t)  \,F_0 \Vert_{L^{2}(m)}\leq C\,e^{bt}\, t^{-\Theta}\,\left\Vert F_0 \right\Vert_{L^{p}(m)}.
\label{duality 2}
\end{align}
To obtain the result, we notice that
$$\Vert \mathcal{A}\,S_{\mathcal{B}}(t)\, F_0\Vert_{L^{2}(m_0)}\leq \Vert \mathcal{A}\,S_{\mathcal{B}}(t)\, F_0\Vert_{L^{2}(m)}, $$
and we combine the previous estimate with the estimate \eqref{duality 2}, which completes the proof of the first estimate.
\end{proof}
\end{cor}
Now we prove Theorem \ref{thm 5}.
\begin{proof}[Proof of Theorem \ref{thm 5}] For $ p \in [1,2] $. 
We consider $ \mathcal{E} = L^{p} (m) $, $ E = L^{2} (m_{0}) $, and denote $ \mathcal{L}_{0} $ and $ L_{0} $ the Fokker-planck operator considered respectively on $ \mathcal{E} $ and $ E $ (defined in \eqref{eq 36}). We split the operator as $\mathcal{L}_{0}=\mathcal{A}+\mathcal{B}$ as in \eqref{def *}. Let us proceed step by step:\\

 \textbf{$\bullet$ Step 1: Verification of condition $ (H_1) $ of Theorem \ref{thm 9}}  \\
 Theorem \ref{thm 2} shows us the existence of the semi-group $ S_{L_{0}} (t) $, associated with the Fokker-Planck operator defined in \eqref{eq 36} on the space $ L^{2} (m_{0})$ and the constants $\kappa$ and $c>0$, for which, for all $  F_{0} \in L^{2} (m) $ such  that $ \langle F_{0} \rangle = 0 $,
 \begin{align}
\forall t \geq 0, \quad  \Vert F(t) \Vert_{L^{2}(m_{0})} \leq c e^{-\kappa t} \Vert F_{0} \Vert_{L^{2}(m_{0})}.
\label{13}
\end{align}
Which implies the dissipativity of the operator $ L_0 -a $ on $ E $, for all $0>a>-\kappa$.\\

 \textbf{$\bullet$ Step 2: Verification of condition $(H_2)$ of Theorem \ref{thm 9}.}\\
  According to Lemma \ref{lem 10}, the operator $\mathcal{B}-a$ is dissipative on $\mathcal{E}$, for all $0>a>a_{m,p}$, and by definition of the operator $\mathcal{A}$ and $A$, we have $\mathcal{A}\in B(\mathcal{E})$ and $A\in B(E)$.\\

 \textbf{$\bullet$ Step 3: Verification of condition $(H_3)$ of Theorem \ref{thm 9}.}

According to Corollary \ref{cor 1}, the operators $\mathcal{A}S_{\mathcal{B}}$ and $S_{\mathcal{B}}\mathcal{A}$ satisfy the property $c)$ of Lemma \ref{lem 8}. By applying Lemma \ref{lem 8}, 
$$\Vert S_{\mathcal{B}}(t)\mathcal{A}\Vert_{B(\mathcal{E},E)} \leq Ce^{bt}\,t^{-\Theta} \text{ and } \Vert \mathcal{A} S_{\mathcal{B}}(t)\Vert_{B(\mathcal{E},E)} \leq Ce^{bt}\,t^{-\Theta}.$$
Then for all $ a'> a $, there exist constructible constants $ n \in \mathbb{N} $ and $ C_{a'} \geq 1 $ , such that
$$ \forall t \geq 0, \quad \Vert T_{n}(t) \Vert_{B(\mathcal{E}, E)} \leq C_{a'} e^{a't}.$$ 
\textbf{$\bullet$ Step 4: End of the Proof }\\
All the hypotheses of Theorem \ref{thm 9}  are satisfied. We deduce that $ \mathcal{L}_{0} -a $ is a dissipative operator on $ \mathcal{E} $ for all $ a> \max (a_{m, p}, -\kappa) $, with the semi-group $ S_{\mathcal{L}_{0}} (t) $ satisfying estimate \eqref{14}.

\end{proof}
\subsubsection{Proof of Theorem \ref{thm 6}}
This part is dedicated to the proof of the exponential time decay estimates of the semi-group associated with the Cauchy  problem \eqref{1}  with an external magnetic field $ B_e $, with an initial datum in $ \tilde{W}^{1, p} (m) $ defined in \eqref{6bis}.
~\par For the proof of Theorem \ref{thm 6}, we consider  the space $\mathcal{E}=\tilde{W}^{1,p}(m)$ and $E=H^{1}(m_{0})$.
\begin{defi}
We split operator $\mathcal{L}_{0}$ into two pieces and define for all $R,M>0$

\begin{align}
\mathcal{B}u=\mathcal{L}_{0}u -\mathcal{A}u \text{ with } \mathcal{A}u=M\chi_{R}u,
\label{def **}
\end{align}
 where $ M > 0$, $\chi_R (v) = \chi (v/R)$ $R > 1,$ and $  \chi \in C^{\infty}_{0}(\mathbb{R}^{3} ) \text{ such that  } \chi (v) = 1$ $ |v| \leq  1.$
 We also denote $A$ and $B$ the restriction of operators $\mathcal{A}$ and $\mathcal{B}$ on the space $E$ respectively. 
 \end{defi}

\begin{lem}[Dissipativity of $\mathcal{B}$]\label{lem 12}
Under Assumptions \ref{hyp 1} and \ref{hyp 2}, there exists $M$ and $R>0$ such that for all $ 0>a >\max ( a_{m,1}^{i},a_{m,2}^{i})$ (defined in \eqref{def 27}-\eqref{def 29} and \eqref{def 21}-\eqref{def 23}) such that operator $\mathcal{B}-a$ is  dissipative in $\tilde{W}^{1,p}(m)$ where $p\in[1,2]$. In other words, the semi-group $\mathcal{S}_{\mathcal{B}}$ satisfies the following estimate:
$$ \forall t\geq 0, \quad\Vert \mathcal{S}_{\mathcal{B}}(t) F_{0}\Vert_{\tilde{W}^{1,p}(m)} \leq e^{at}\,\Vert F_{0}\Vert_{\tilde{W}^{1,p}(m)}, \quad \forall F_{0}\in \tilde{W}^{1,p}(m).$$
\begin{proof}
Let $ F_{0} \in \tilde{W}^{1, p} (m) $. We consider $ F $ the solution of the evolution equation
\begin{align}\label{eq 40}
\partial_{t}F=\mathcal{B}F,\quad F_{|_{t=0}}=F_0.
\end{align}
Recall that the norm on the space $ \tilde{W}^{1, p} (m) $ is given by
\begin{align*}
\Vert F \Vert_{\tilde{W}^{1,p}(m)}^{p} =  \Vert F \Vert_{L^{p}(\tilde{m})}^{p}+ \Vert \nabla_{v}F\Vert_{L^{p}(m)}^{p}+ \Vert \nabla_{x}F\Vert_{L^{p}(m)}^{p},
\end{align*}
where $\tilde{m}=m\langle v\rangle$.
Differentiating the previous equality with respect to $t$, we get
\begin{align}\label{eg 0}
\frac{d}{dt}\frac{1}{p}\Vert F \Vert_{\tilde{W}^{1,p}(m)}^{p} =  \frac{d}{dt}\frac{1}{p}\Vert F\Vert_{L^{p}(\tilde{m})}^{p}+ \frac{d}{dt}\frac{1}{p}\Vert \nabla_{v}F\Vert_{L^{p}(m)}^{p}+\frac{d}{dt}\frac{1}{p} \Vert \nabla_{x}F\Vert_{L^{p}(m)}^{p}.
\end{align}
We now estimate each term of the equality \eqref{eg 0}.
~\par For the first term in \eqref{eg 0}, we apply Lemma \ref{lem 10} and get
\begin{align*}
\displaystyle\frac{1}{p}\frac{d}{dt}\, \Vert F\Vert_{L^{p}(\tilde{m})}^{p} 
\leq \iint_{ \mathbb{T}^{3}\times \mathbb{R}^{3}}\, \vert F \vert^{p} (\Psi_{\tilde{m},p} - M\chi_{R})\, \tilde{m}^{p}(v) \,dxdv,
\end{align*}
Secondly, we differentiate the equation \eqref{eq 40} with respect to $ v $, and then we use the equalities of Lemma \ref{lem 3}. We get the following equation (recall $d=3$):
\begin{align}
\partial_{t}\nabla_{v}F=\mathcal{B}(\nabla_{v}F)+3\nabla_{v}F+(B_{e}\wedge \nabla_{v} )F-\nabla_{x}F-M(\nabla_{v}\cdot\chi_{R})F_{t}.
\end{align}
This gives
\begin{align*}
\frac{d}{dt}\frac{1}{p}\Vert \nabla_{v}F\Vert_{L^{p}(m)}^{p} &=\iint\partial_{t}\nabla_{v}F \vert \nabla_{v}F\vert^{p-2}\cdot\nabla_{v}F\,m^{p}\,dxdv\\
&=\iint \mathcal{B}(\nabla_{v}F)\vert \nabla_{v}F\vert^{p-2}\cdot\nabla_{v}F\,m^{p}\,dxdv+3\Vert \nabla_{v}F\Vert_{L^{p}(m)}^{p}\\
&-\iint \nabla_{x}F\vert \nabla_{v}F\vert^{p-2}\cdot\nabla_{v}F\,m^{p}\,dxdv\\
&+\iint (B_{e}\wedge \nabla_{v})F\, \vert \nabla_{v}F\vert^{p-2}\cdot\nabla_{v}F\,m^{p}\,dxdv\\
&-M\iint(\nabla_{v}\chi_{R})F\,\vert \nabla_{v}F\vert^{p-2}\cdot\nabla_{v}F\,m^{p}\,dxdv.
\end{align*}
Then, proceeding exactly as in the proof of Lemma \ref{lem 10} and applying Young's inequality, we obtain for all $\eta_{1}>0$
\begin{align*}
&\frac{d}{dt}\frac{1}{p}\Vert \nabla_{v}F\Vert_{L^{p}(m)}^{p} \\
&\leq \iint \vert\nabla_{v}F\vert^{p} (\Psi_{m,p} - M\chi_{R})\, m^{p} \,dxdv+3\Vert \nabla_{v}F\Vert_{L^{p}(m)}^{p}\\
&+\frac{1}{2}\Vert\nabla_{x}F\Vert_{L^{p}(m)}^{p}+\frac{1}{2}\Vert \nabla_{v}F\Vert_{L^{p}(m)}^p+\frac{M}{R}C_{\eta_{1}}\Vert \nabla_{v}\, \chi\Vert_{L^{\infty}(\mathbb{R}^{3})}\,\Vert F\Vert_{L^{p}(m)}^{p}\\
&+\frac{M}{R}\eta_{1}\Vert \nabla_{v}\,\chi\Vert_{L_{\infty}(\mathbb{R}^{3})}\,\Vert \nabla_{v}F\Vert_{L^{p}(m)}+\Vert B_{e}\Vert_{L^{\infty}(\mathbb{T}^{3})}\,\Vert \nabla_{v}F\Vert_{L^{p}(m)}^{p}\\
&\leq \iint \left(\vert \nabla_{v}F\vert^{p}\,(\Psi_{m,p} -M\chi_{R}+3+\frac{1}{2} 
+\frac{M}{R}\Vert \nabla_{v} \,\chi\Vert_{L^{\infty}(\mathbb{R}^{3})}\eta_{1}+\Vert B_{e}\Vert_{L^{\infty}(\mathbb{T}^{3})}\right)\,m^{p}\,dxdv\\
&+\frac{1}{2}\Vert\nabla_{x}F\Vert_{L^{p}(m)}^{p}+\frac{M}{R}C_{\eta_{1}}\Vert \nabla_{v}\, \chi\Vert_{L^{\infty}(\mathbb{R}^{3})}\,\Vert F\Vert_{L^{p}(m)}^{p}.
\end{align*}
Finally, we estimate the last term of the equality \eqref{eg 0}. 
 We treat two cases, and then we use an interpolation argument to complete the proof.\\
\textbf{$\bullet$ Case 1: $p=1$.}\\
We differentiate the equation \eqref{eq 40} with respect to $ x_i $ for all $i=1,2,3$, then we use the equalities of Lemma \ref{lem 3}. We will have the following  equation:
\begin{align}
\partial_{t}\partial_{x_i}F=\mathcal{B}(\partial_{x_i}F)+(v\wedge\partial_{x_i}B_{e})\cdot\nabla_{v} F. 
\label{equation 8}
\end{align}
Using the previous equation, we obtain
\begin{align*}
\frac{d}{dt} \,\Vert \partial_{x_{i}}F\Vert_{L^{1}(m)}& = \iint \,\partial_{t} \,\vert \partial_{x_{i}}F\vert\,m\,dxdv\\
&=\iint (\partial_{x_i}\partial_{t}F)\, \partial_{x_i}F\, \vert \partial_{x_i}F\vert^{-1}\,
m\,dxdv \\
&=\iint \, \mathcal{B}(\partial_{x_i}F)\,\partial_{x_i}F\vert \partial_{x_i}F\vert\,m \, dxdv\\
&+\iint (v\wedge \partial_{x_i}B_e)\cdot \nabla_{v}F\,\partial_{x_i}F\,\vert\partial_{x_i}F\vert^{-1}\, m \,dxdv.
\end{align*}
Using the computations made in Lemma \ref{lem 10} for $p=1$, using Lemma \ref{lem *} in the appendix B, and performing an integration by parts with respect to $v$, we get
\begin{align*}
&\frac{d}{dt} \,\Vert \partial_{x_{i}}F\Vert_{L^{1}(m)}\\
& \leq  \iint \left( \Psi_{m,1} - M\chi_{R}\right) \,  \vert\partial_{x_i}\,F\vert
\,m\,dxdv - \underbrace{\iint (v\wedge \partial_{x_i}B_e)F\,\partial_{x_i}F\,\vert \partial_{x_i}F\vert^{-1}\,\nabla_{v} m\,dxdv}_{=0},
\end{align*}
where, we used the fact that $(v\wedge \partial_{x_i} B_e )\cdot \nabla_{v}m=0$. Then, defining the norm 
$$\Vert \nabla_{x}F\Vert_{L^{p}(m)}:=\sum_{i=1}^{3}\,\Vert \partial_{x_i} F\Vert_{L^{p}(m)}, $$
and using the previous definition, we have
$$\frac{d}{dt} \,\Vert \nabla_x F\Vert_{L^{1}(m)} \leq  \iint \left( \Psi_{m,1} - M\chi_{R}\right) \, \vert \nabla_x \,F\vert \,m\,dxdv . $$
Collecting all the estimates, we obtain
\begin{align*}
&\frac{d}{dt} \Vert F\Vert_{\tilde{W}^{1,1}(m)}\\
&\leq \iint\left( \Psi_{\tilde{m},1} -M\chi_{R}+\frac{M}{R}C_{\eta_{1}}\Vert \nabla_{v} \chi\Vert_{L^{\infty}(\mathbb{R}^{3})} \right)\vert F\vert\,\tilde{m}\,dxdv\\
&+\iint \left( \Psi_{m,1} -M\chi_{R}+3+\frac{1}{2}+\frac{M}{R}\eta_{1}\Vert \nabla_{v}\chi\Vert_{L^{\infty}(\mathbb{R}^{3})}+\Vert B_{e}\Vert_{L^{\infty}(\mathbb{T}^{3})}\right) \vert \nabla_{v}F\vert\,m\,dxdv\\
&+\iint \left( \Psi_{m,1}-M\chi_{R}+\frac{1}{2}\right)\vert \nabla_{x}F\vert\, m\,dxdv. 
\end{align*}
We define then (for $M$ and $R$ to be fixed below).
\begin{align}
\Psi_{m,1}^{1}& :=\Psi_{\tilde{m},1} -M\chi_{R}+\frac{M}{R}C_{\eta_{1}}\Vert \nabla_{v} \chi\Vert_{L^{\infty}(\mathbb{R}^{3})},\label{def 24}\\
\Psi_{m,1}^{2}&:=  \Psi_{m,1} -M\chi_{R}+3+\frac{1}{2}+\frac{M}{R}\eta_{1}\Vert \nabla_{v} \chi\Vert_{L^{\infty}(\mathbb{R}^{3})}+\Vert B_{e}\Vert_{L^{\infty}(\mathbb{T}^{3})},\label{def 25}\\
\Psi_{m,1}^{3}&:= \Psi_{m,1}-M\chi_{R}+\frac{1}{2}.
\label{def 26}
\end{align}
(Recall that $\limsup\limits_{\vert v\vert \to +\infty} \Psi_{m,1}=-k$). We denote then
\begin{align}
a_{m,1}^{1}&=  -k-1 ,\label{def 27}\\
a_{m,1}^{2}& =  -k +\frac{7}{2}+\Vert B_e\Vert_{L^{\infty}(\mathbb{T}^{3})},\label{def 28}\\
a_{m,1}^{3}& =  -k +\frac{1}{2}.
 \label{def 29}
 \end{align}   
We now assume that $k$ satisfies
\begin{align}
\label{hypothese 5}
k>\frac{7}{2}+\Vert B_e \Vert_{L^{\infty}(\mathbb{T}^{3})}.
\end{align}
Hypothesis \eqref{hypothese 5} implies that $a_{m,1}^{i}<0$, for all $i=1,2,3$. Consequently, for  $\eta_{1}$ sufficiently small, we may then find  $M$ and $R>0$ large enough so that, for all $0>a>\max (a_{m,1}^{1},a_{m,1}^{2},a_{m,1}^{3})$, we have 
\begin{align}
\label{conc 1}
\frac{d}{dt}\, \Vert F(t)\Vert_{\tilde{W}^{1,1}(m)} \leq a \Vert F(t)\Vert_{\tilde{W}^{1,1}(m)}.
\end{align}
Hence the operator $\mathcal{B}-a$ is dissipative on $\tilde{W}^{1,1}(m)$ .\\
\textbf{$\bullet$ Case 2: $p=2$. }\\
Again, we differentiate the equation \eqref{eq 40} with respect to $ x$, and we use the equalities of Lemma \ref{lem 3} to obtain the following  equation:
\begin{align}
\partial_{t}\nabla_{x}F=\mathcal{B}(\nabla_{x}F)+(v\wedge\nabla_{x}B_{e})\cdot\nabla_{v} F. 
\label{equation 9}
\end{align}
Using the calculations made in Lemma \ref{lem 10} and the previous equation, we obtain
\begin{align*}
\frac{d}{dt}\frac{1}{2}\Vert \nabla_{x}F\Vert_{L^{2}(m)}^{2} &=-\iint\vert \nabla_{v}\nabla_{x}F\vert^{2}m^{2}\,dxdv\\
&+\iint (\Psi_{m,2} -M\chi_{R})\vert \nabla_{x}F\vert^{2}m^{2}\,dxdv \\
& +\iint (v\wedge\nabla_{x}B_{e})\cdot\nabla_{v}F\,\nabla_{x}F\,m^{2}\,dxdv.
\end{align*} 
Then, by integration by parts with respect to $v$, we get 
\begin{align*}
\frac{d}{dt}\,\frac{1}{2}\Vert \nabla_{x}F\Vert_{L^{2}(m)}^{2} &\leq -\iint \vert \nabla_{v}\nabla_{x}F\vert^{2}m^{2}\,dxdv\\
&+\iint (\Psi_{m,2} -M\chi_{R})\vert \nabla_{x}F\vert^{2}m^{2}\,dxdv \\& +\iint \vert v\wedge \nabla_{x} B_e\vert \vert F\vert \, \vert \nabla_{x}\nabla_{v}F\vert\,m^{2}\,dxdv
\end{align*}
According to the Cauchy-Schwarz inequality, for every $ \varepsilon> 0 $, there is a $C_{\varepsilon}>0$ such that
\begin{align*}
\frac{d}{dt}\frac{1}{2}\Vert \nabla_{x}F\Vert_{L^{2}(m)}^{2} &\leq -\iint\vert \nabla_{v}\nabla_{x}F\vert^{2}m^{2}\,dxdv+\iint (\Psi_{m,2} -M\chi_{R})\vert \nabla_{x}F\vert^{2}m^{2}\,dxdv \\
&+\varepsilon\iint \vert \nabla_{v}\nabla_{x}F\vert^{2}\,m^{2}\,dxdv+C_{\varepsilon}\iint \vert v\wedge B_{e}\vert^{2}\vert F\vert^{2}\,m^{2}\,dxdv.
\end{align*}
We choose $ \varepsilon =\dfrac{1}{4} $, and we finally get
\begin{align*}
\frac{d}{dt}\frac{1}{2}\Vert \nabla_{x}F\Vert_{L^{2}(m)}^{2} &
\leq \iint (\Psi_{m,2} -M\chi_{R})\vert \nabla_{x}F\vert^{2}m^{
2}\,dxdv\\
&+\frac{1}{2}\Vert \nabla_{x}B_{e}\Vert_{L^{\infty}(\mathbb{T}^{3})}\Vert F\Vert_{L^{2}(\tilde{m})}^{2}+\frac{1}{2}\Vert \nabla_{x}F\Vert_{L^{2}(m)}^{2}.
\end{align*}
Collecting all the estimates, we thus obtain
\begin{align*}
&\frac{d}{dt}\frac{1}{2}\Vert F\Vert_{\tilde{W}^{1,2}(m)}^{2}\\
&\leq \iint\left( \Psi_{\tilde{m},2} -M\chi_{R}+\frac{M}{R}C_{\eta_{1}}\Vert \nabla_{v}\, \chi\Vert_{L^{\infty}(\mathbb{R}^{3})}+\frac{1}{2}\Vert \nabla_{x}B_{e}\Vert_{L^{\infty}(\mathbb{T}^{3})}\right)\vert F\vert^{2}\,\tilde{m}^{2}\,dxdv\\
&+\iint \left( \Psi_{m,2} -M\chi_{R}+3+\frac{1}{2}+\frac{M}{R}\eta_{1}\Vert \nabla_{v}\, \chi\Vert_{L^{\infty}(\mathbb{R}^{3})}+\Vert B_{e}\Vert_{L^{\infty}(\mathbb{T}^{3})}\right) \vert \nabla_{v}F\vert^{2}\,m^{2}\,dxdv\\
&+\iint \left( \Psi_{m,2}-M\chi_{R}+\frac{1}{2}+\frac{1}{2}\right)\vert \nabla_{x}F\vert^{2}\, m^{2}\,dxdv
\end{align*}
Again, we define then, for $M$ and $R$ to be fixed in the next paragraph
\begin{align}
\Psi_{m,2}^{1}& =\Psi_{\tilde{m},2} -M\chi_{R}+\frac{M}{R}C_{\eta_{1}}\Vert \nabla_{v}\, \chi\Vert_{L^{\infty}(\mathbb{R}^{3})}+\frac{1}{2}\Vert \nabla_{x}B_{e}\Vert_{L^{\infty}(\mathbb{T}^{3})},\label{def 18}\\
\Psi_{m,2}^{2}&=  \Psi_{m,2} -M\chi_{R}+3+\frac{1}{2}+\frac{M}{R}\eta_{1}\Vert \nabla_{v}\, \chi\Vert_{L^{\infty}(\mathbb{R}^{3})}+\Vert B_{e}\Vert_{L^{\infty}(\mathbb{T}^{3})},\label{def 19}\\
\Psi_{m,2}^{3}&= \Psi_{m,2}-M\chi_{R}+1.
\label{def 20}
\end{align}
Again, we denote  
\begin{align}
a_{m,2}^{1}& = \frac{3}{2} -k-1 +\frac{1}{2}\Vert \nabla_{x}B_{e}\Vert_{L^{\infty}(\mathbb{T}^{3})}\label{def 21}\\
a_{m,2}^{2}&= \frac{3}{2} -k +\frac{7}{2}+\Vert B_e\Vert_{L^{\infty}(\mathbb{T}^{3})}\label{def 22}\\
a_{m,2}^{3}& = \frac{3}{2} -k +1,
 \label{def 23}
 \end{align}  
Assuming $ k $ satisfies
\begin{align}
k>5+ \max\left(\Vert B_e\Vert_{L^{\infty}(\mathbb{T}^{3})},\frac{1}{2}\Vert \nabla_{x}B_e\Vert_{L^{\infty}(\mathbb{T}^{3})}\right),
\label{hypothese 6}
\end{align}
 we obtain that $a_{m,2}^{i}<0$ for all $ i=1,2,3$. Consequently, we may find $M$ and $R>0$ large enough so that for all $0>a>\max (a_{m,2}^{1},a_{m,2}^{2},a_{m,2}^{3})$
$$ \frac{d}{dt}\frac{1}{2}\Vert F(t)\Vert_{\tilde{W}^{1,2}(m)}^{p}\leq  a\, \Vert F(t)\Vert_{\tilde{W}^{1,2}(m)}^{2}. $$
 Hence the operator $\mathcal{B}-a$  is dissipative on $\tilde{W}^{1,2}(m)$ for such $M$ and $R$. 
 ~\par \textbf{ For the general case  $1\leq p\leq 2 $:}
The cases $ 1 $ and $ 2 $ show us that the operator $ S_{\mathcal{B}} (t) $ is continuous on $ \tilde{W}^{1,1} (m) $ (on $ \tilde{W}^{1,2} (m) $) with the operator $ \mathcal{B} $ is given by
 $$\mathcal{B}=\mathcal{L}_{0}-M\,\chi_{R}, $$
where $M$ and $R>0$ agree with the conditions given in case $1$ and case $2$.
Applying the Riesz-Thorin interpolation Theorem and using Hypothesis \ref{hyp 2},  we obtain that the operator $ S_{\mathcal{B}} (t) $ is continuous on $ \tilde{W}^{1, p} (m) $ for all $ 1 \leq p \leq $ 2, with the following dissipative estimate:
 $$ \forall 0>a>\max (a_{m,1}^{i}, a^{i}_{m,2},\, i=1,2,3), \quad \Vert S_{\mathcal{B}}(t) \,F_0 \Vert_{\tilde{W}_{1,p}(m)}\leq Ce^{at}\, \Vert F_0\Vert_{\tilde{W}^{1,p}(m)}. $$
This completes the proof.
\end{proof}
\end{lem}

From now on, $M$ and $R$ are fixed.
\begin{lem}[Property of regularization]\label{lem 13}
There exist $b$ and $C > 0$ such that, for all $p,q$ with $ 1\leq p\leq q \leq 2,$  we have
\begin{align}
\forall t\geq 0,\quad \Vert S_{\mathcal{B}}(t) F_{0}\Vert_{\tilde{W}^{1,q}(m_{0})} \leq Ce^{bt}t^{-(3d+4)(\frac{1}{p}-\frac{1}{q})}\, \Vert F_{0}\Vert_{\tilde{W}^{1,p}(m_{0})}.
\label{eg 47}
\end{align}
\begin{align}
\forall t\geq 0,\quad \Vert S_{\mathcal{B}^{*}}(t) F_{0}\Vert_{\tilde{W}^{-1,p^{'}}(m_{0})} \leq Ce^{bt}t^{-(3d+4)(\frac{1}{p}-\frac{1}{q})}\, \Vert F_{0}\Vert_{\tilde{W}^{-1,q^{'}}(m_{0})}.
 \label{eg 48}
\end{align}
Here  $2\leq q^{'}\leq p^{'}\leq +\infty $ are the conjugates of $p$ and $q$ respectively.
\begin{proof}
Let $ F $ be the solution of the  evolution equation
$$ \partial_{t}F=\mathcal{B}F,\quad F_{|_{t=0}}=F_{0}.$$
In to the proof of Lemma \ref{lem 11}, the following relative entropy has been introduced
\begin{align*}
\mathcal{H}(t,h) = B\Vert h\Vert_{L^{1}(m_{0})}^{2}+t^{r}\mathcal{G}(t,h),
\end{align*}
with $ \mathcal{G} $ defined by
$$\mathcal{G}(t,h)=C \Vert h \Vert^{2}_{L^{2}(m_{0})} + D\, t \Vert \nabla_{v} h \Vert^{2}_{L^{2}(m_{0})}  + E \, t^{2} \langle \nabla_{x}h, \nabla_{v}h \rangle + a\,t^{3}\, \Vert \nabla_{x} h\Vert^{2}_{L^{2}(m_{0})}. $$
We have shown, for constants $ \alpha, D, E $ and $ \beta > 0 $ well chosen, that there exist $ C> 0 $ and $ r = 3d + 1 $ such that
$$\forall t\geq 0, \quad\mathcal{H}(t,F)\leq  B'''\,\mathcal{H}(0,F_{0})\leq C\Vert F_0\Vert^{2}_{L^1 (m_0)}.  $$
Using the previous estimate and the definition of $ \mathcal{H} $ and $ \mathcal{G} $, we get
\begin{align*}
\Vert S_{\mathcal{B}}(t)F_{0}\Vert_{L^{2}(m_{0})}^{2}&\leq \frac{\alpha}{t^{3d+1}}\, \mathcal{H}(t, F) \leq \frac{C'}{t^{3d+1}}\,e^{bt}\Vert F_{0}\Vert_{L^{1}(m_{0})}^{2},\\
\Vert \nabla_{v} S_{\mathcal{B}}(t)F_{0}\Vert_{L^{2}(m_{0})}^{2}&\leq \frac{D}{t^{3d+2}}\, \mathcal{H}(t, F)\leq \frac{C''}{t^{3d+2}}\,e^{bt}\Vert F_{0}\Vert_{L^{1}(m_{0})}^{2},\\
\Vert \nabla_{x}S_{\mathcal{B}}(t)F_{0}\Vert_{L^{2}(m_{0})}^{2}&\leq \frac{\beta}{t^{3d+4}}\, \mathcal{H}(t, F)\leq \frac{C'''}{t^{3d+4}}\,e^{bt}\Vert F_{0}\Vert_{L^{1}(m_{0})}^{2}.
\end{align*}
Therefore,
$$ \forall t\in [0,+\infty ),\quad\Vert S_{\mathcal{B}}(t)F_{0}\Vert_{H^{1}(m_{0})}^{2} \leq \frac{\tilde{C}}{t^{3d+4}}e^{bt}\,\Vert F_{0}\Vert_{W^{1,1}(m_{0})}^{2}.$$
Finally, to complete the proof, we use the Riesz-Thorin Interpolation Theorem in the real case on the operator $ S_{\mathcal{B}}(t) $ . We obtain the continuity of $ S_{\mathcal {B}} (t)$ from $ \tilde {W}^{1, p} (m_{0}) $ to $ \tilde{W}^{1, q} (m_{0}) $, with $ S_{\mathcal{B}} $ satisfying the estimate \eqref{eg 47}.
~\par The estimate \eqref{eg 48} follows from \eqref{eg 47} by duality.
\end{proof}
\end{lem}

\begin{cor}\label{cor 2}
Let  $m$ be a weight that satisfies Hypothesis \ref{hyp 2}. Then there exists $\Theta \geq 0$ such that for all $F_{0}\in \tilde{W}^{1,p}(m)$ with  $p\in [1,2]$ 
\begin{align*}
\forall t\geq 0, \quad \Vert \mathcal{A}S_{\mathcal{B}}(t) F_0\Vert_{H^{1}(m_{0})} &\leq Ce^{bt}\, t^{-\Theta}\, \Vert F_0\Vert_{\tilde{W}^{1,p}(m)},\\
 \forall t\geq 0,\quad\Vert S_{\mathcal{B}} \mathcal{A}(t)F_0\Vert_{H^{1}(m_{0})}& \leq Ce^{bt}\, t^{-\Theta}\, \Vert F_0\Vert_{\tilde{W}^{1,p}(m)} .
\end{align*}
\begin{proof}
The proof is similar to that of Corollary \ref{cor 1}.
\end{proof}
\end{cor}
\begin{proof}[Proof of Theorem \ref{thm 6}]
The estimate \eqref{eg 8} is an immediate consequence of Theorem \ref{thm 9} together with Theorem \ref{thm 4}, Lemma \ref{lem 12}, Lemma \ref{lem 13} and Corollary \ref{cor 2}.
\end{proof}
\appendix

\section{Property of the operator $ (v \wedge B_e) \cdot \nabla_v $}
In the following Lemma we show that operator $ (v \wedge B_e) \cdot \nabla_v $ and $v\cdot\nabla_{x}$ are  formally skew-adjoint operators in the space $ L^{2} (dxd \mu) $.
\begin{lem}\label{lem 14}
Let $B_{e}$ be the external magnetic field, then, with adjoints in the space $L^{2}(dxd\mu)$,
$$ ( (v\wedge B_{e})\cdot \nabla_{v} )^{*} = - (v\wedge B_{e})\cdot \nabla_{v}, $$
and 
$$\left(v\cdot\nabla_{x}\right)^{*}=-v\cdot\nabla_{x}. $$
 \begin{proof}
 Let $f$ and $g$ $\in C^{\infty}_{0}(\mathbb{R}^{3} \times \mathbb{T}^{3})$. We have 
 \begin{align*}
 \langle ((v\wedge B_{e})\cdot \nabla_{v} )^{*} f , g \rangle & = \langle f, ((v\wedge B_{e})\cdot \nabla_{v})g \rangle \\
 & = \iint\, f  \,((v\wedge B_{e})\cdot \nabla_{v})g dx d\mu.
 \end{align*}
  Using the fact $$(v \wedge B_{e})\cdot \nabla_{v}f = \nabla_{v}\cdot (v \wedge B_{e})f,$$
we obtain 
 \begin{align*}
 \langle ((v\wedge B_{e})\cdot \nabla_{v} )^{*} f , g \rangle  & = \iint f \,  \nabla_{v}\cdot(v\wedge B_{e})g \, dx d\mu \\
 &= - \iint f (-\nabla_{v} + v )\cdot(v\wedge B_{e})\,g \, dx d\mu,
 \end{align*}
since $(v \wedge B_{e})\cdot v = 0 $. By integration by parts,  we have then
\begin{align*}
\langle ((v\wedge B_{e})\cdot \nabla_{v} )^{*} f , g \rangle & = - \iint \nabla_{v}f \cdot (v \wedge B_{e})g \, dx d\mu \\
&= - \iint g \,(v\wedge B_{e})\cdot \nabla_{v}f \, dx d\mu \\
&= -\langle ((v\wedge B_{e} )\cdot \nabla_{v}) f, g \rangle.
\end{align*}
For the second equality, we obtain
\begin{align*}
\langle \left( v\cdot\nabla_{x}\right)^{*} f, g\rangle = \langle f,  v\cdot\nabla_{x}g\rangle
=\iint f\,\left(v\cdot\nabla_{x}g\right)\,dxd\mu,
\end{align*}
by integration by parts with respect to $x$. Since $\mu$ is independent of $x$, we have then 
\begin{align*}
\langle \left( v\cdot\nabla_{x}\right)^{*} f, g\rangle=-\iint \left(v\cdot\nabla_{x}f\right)\,g\,dxd\mu 
=-\langle v\cdot\nabla_{x}f, g\rangle.
\end{align*}
This completes the proof.
 \end{proof}
   \end{lem}
 \begin{remarque}\label{remarque A.2}
 We can generalize the results of the preceding Lemma for all function $m$ which are radial in $v$ and independant of $x$. We obtain that $v\cdot\nabla_{x}$ and ($v\wedge B_e)\cdot\nabla_{v}$ are formally skew-adjoint operators in the space $L^2(m)$.
 \end{remarque}

 \section{Non positivity of a certain integral} The following well-know lemma is used in the proof of the dissipativity of the operator $\mathcal{B}-a$ in the spaces $L^p (m)$ and $\tilde{W}^{1,p}(m)$ in Section \ref{section 4}. This lemma is a special case of the general study done in the article \cite{chafai2004entropies}.
 \begin{lem}\label{lem *}
 Let $g$ be a smooth function and let $p\geq 1$. Then the following integral is well-posed and satisfy the following estimate
 \begin{align*}
 \iint_{\mathbb{T}^{3}\times\mathbb{R}^{3}}\, \left(\Delta_{v}g \right)\vert g\vert^{p-2}\,g\,dxdv \leq 0.
 \end{align*}
 \begin{proof}
Formal integration by parts with respect to $v$ justifies the property for all $p>1$.   For $p=1$, we regularize and use  convexity of the function $\Psi: s\to \vert s\vert$.
 \end{proof}
 \end{lem}

\section*{Acknowledgments} I would like to thank my advisors Fr\'ed\'eric H\'erau and Joseph Viola for their invaluable help and support during the maturation of this paper. I thank also the Centre Henri Lebesgue ANR-11-LABX-0020-01 and the Faculty of Sciences (Section I) of Lebanese University for its support .



\end{document}